\documentclass[3p]{elsarticle}
\usepackage{bm}
\usepackage[parfill]{parskip}
\usepackage[utf8]{inputenc}
\usepackage{url}
\usepackage{amsmath,amssymb,amsfonts,amsthm}
\usepackage{color}
\usepackage{bm}
\usepackage{graphicx}
\usepackage{subcaption}
\usepackage{lipsum}
\usepackage{amsfonts}
\usepackage{graphicx}
\usepackage{epstopdf}
\usepackage{array}
\usepackage[hidelinks]{hyperref}
\usepackage{lineno}
\usepackage{bbm}
\usepackage{xfrac}

\usepackage{amsmath}
\usepackage{algorithm}
\usepackage[noend]{algpseudocode}

\ifpdf
  \DeclareGraphicsExtensions{.eps,.pdf,.png,.jpg}
\else
  \DeclareGraphicsExtensions{.eps}
\fi

\newtheorem{theorem}{Theorem}

\newtheorem{remark}{Remark}

\makeatletter
\newsavebox\myboxA
\newsavebox\myboxB
\newlength\mylenA

\newcommand*\xoverline[2][0.75]{%
    \sbox{\myboxA}{$\m@th#2$}%
    \setbox\myboxB\null% Phantom box
    \ht\myboxB=\ht\myboxA%
    \dp\myboxB=\dp\myboxA%
    \wd\myboxB=#1\wd\myboxA% Scale phantom
    \sbox\myboxB{$\m@th\overline{\copy\myboxB}$}%  Overlined phantom
    \setlength\mylenA{\the\wd\myboxA}%   calc width diff
    \addtolength\mylenA{-\the\wd\myboxB}%
    \ifdim\wd\myboxB<\wd\myboxA%
       \rlap{\hskip 0.5\mylenA\usebox\myboxB}{\usebox\myboxA}%
    \else
        \hskip -0.5\mylenA\rlap{\usebox\myboxA}{\hskip 0.5\mylenA\usebox\myboxB}%
    \fi}
\makeatother

\journal{arXiv.org}

\newcommand{\TheTitle}{On the stability of robust dynamical low-rank approximations for hyperbolic problems} 

\date{\today}

\usepackage{amsopn}

\newcommand{\RomanNumeralCaps}[1]
    {\MakeUppercase{\romannumeral #1}}

\journal{arXiv}

\begin{document}
%\linenumbers
\begin{frontmatter}

\title{\TheTitle}

\author[adressJonas]{Jonas Kusch}
\author[adressLukas]{Lukas Einkemmer}
\author[adressGianluca]{Gianluca Ceruti}

\address[adressJonas]{Karlsruhe Institute of Technology, Karlsruhe,jonas.kusch@kit.edu}
\address[adressLukas]{University of Innsbruck, Innsbruck, Austria, lukas.einkemmer@uibk.ac.at}
\address[adressGianluca]{Universit{\"a}t T{\"u}bingen, T{\"u}bingen, Germany, ceruti@na.uni-tuebingen.de}

\begin{abstract}
The dynamical low-rank approximation (DLRA) is used to treat high-dimensional problems that arise in such diverse fields as kinetic transport and uncertainty quantification. Even though it is well known that certain spatial and temporal discretizations when combined with the DLRA approach can result in numerical instability, this phenomenon is poorly understood. In this paper we perform a $L^2$ stability analysis for the corresponding nonlinear equations of motion. This reveals the source of the instability for the projector splitting integrator when first discretizing the equations and then applying the DLRA. Based on this we propose a projector splitting integrator, based on applying DLRA to the continuous system before performing the discretization, that recovers the classic CFL condition. We also show that the unconventional integrator has more favorable stability properties and explain why the projector splitting integrator performs better when approximating higher moments, while the unconventional integrator is generally superior for first order moments. Furthermore, an efficient and stable dynamical low-rank update for the scattering term in kinetic transport is proposed. Numerical experiments for kinetic transport and uncertainty quantification, which confirm the results of the stability analysis, are presented.

%Dynamical low-rank approximation (DLRA) for parameterized partial differential equations has recently gained increasing interest in the fields of uncertainty quantification and kinetic transport. Spatial discretizations for DLRA have a significant impact on the solution quality and can lead to instabilities resulting in a break down of the method.
%In this work we propose a low-rank Fourier analysis to investigate the stability of spatial discretizations for dynamical low-rank approximation in the case of linear problems. The analysis reveals favourable stability properties of the unconventional integrator, whereas the matrix projector-splitting integrator requires a carefully chosen spatial discretization to ensure stability. This stabilization is constructed by deriving the DLRA evolution equations of the projector-splitting integrator on the continuous level and constructing spatial discretizations for each individual equation. Furthermore, an efficient and stable dynamical low-rank update for scattering terms for kinetic transport is proposed. Numerical experiments for kinetic transport and uncertainty quantification, which confirm the analytic properties are conducted.
\end{abstract}


\end{frontmatter}

\section{Introduction}

Dynamical low-rank approximation (DLRA) \cite{KochLubich07} for parametrized partial differential equations has gained increasing attention in the last years. This stems mainly from its ability to mitigate the curse of dimensionality in terms of computational costs and memory requirements. Problems in which dynamical low-rank approximation has proven its efficiency include, e.g., kinetic theory \cite{EiL18,einkemmer2019quasi,PeMF20,PeM20,einkemmer2020low,EiHY21,EiJ21,guo2021low} as well as uncertainty quantification \cite{FeL18,MuN18,MuNV20,SaL09,kusch2021DLRUQ}. In both fields the high-dimensional phase space implies that obtaining numerical solutions is extremely expensive both in terms of memory and computational cost.

Robust integrators for the DLRA evolution equations are the matrix projector-splitting integrator, introduced in \cite{LubichOseledets}, as well as the unconventional integrator, introduced in  \cite{CeL21}. Both integrators are unaffected by the presence of small singular values. A main difference of the unconventional integrator is that the dynamics is only moving forward in time, whereas the projector-splitting integrator includes a step, which moves backward. This property plays a key role in the stability for spatial discretizations, which we will see in this work. Furthermore, the unconventional integrator preserves symmetry or anti-symmetry of the original problem \cite{CeL21}. On the other hand, the projector splitting integrator can be extended to second order, which has been widely used in the literature \cite{EiL18,einkemmer2019low,einkemmer2021asymptotic,einkemmer2019quasi}.

It has been shown in numerical experiments that the unconventional integrator yields smoother solution profiles for first order moments such as the scalar flux in radiation transport or the expected value in uncertainty quantification \cite{kusch2021DLRUQ}. Higher-order moments, however, are dampened heavily by the unconventional integrator and the projector-splitting integrator allows for a more adequate representation \cite{kusch2021DLRUQ}. However, the reason for this behavior is not understood, which is the main motivation for this work.

In order to implement a dynamical low-rank integrator, the partial differential equation under consideration has to be discretized. There are two main approaches to determine an approximation to spatial derivatives in the equations of DLRA. First, the spatial discretization can be performed for the original equation, leading to a matrix differential equation to which the dynamical low-rank approximation is applied. Second, the dynamical low-rank approximation can be derived for the continuous problem and the evolution equations of DLRA can be discretized in space in a subsequent step (as has been suggested in \cite{EiL18}). The first approach is extensively used. However, as we will show, it can suffer from instabilities. For the second approach, a set of differential equations is obtained that can be discretized by an appropriate method. This enables the implementation of a suitable stabilization for each individual substep of the two integrators. The construction of adequate stabilization strategies for each substep of the two integrators requires knowledge of dampening and amplification of the underlying dynamics, which we aim to establish in this work.

In this work, we answer the questions
\begin{enumerate}
    \item Is there an analytic explanation why the projector-splitting integrator yields oscillatory first-order moments, while showing a satisfactory approximation for second-order moments compared to the unconventional integrator?
    \item Should dynamical low-rank approximation be derived for the matrix ODE which results from a discretization of the original problem? Or should dynamical low-rank approximation be performed for the continuous problem and the discretization be applied to the continuous DLRA evolution equations?
    \item If the latter option is chosen: How should the time and space discretization for the different substeps be chosen to obtain a stable and accurate numerical method?
\end{enumerate}
The tool that we use to answer these questions is a Fourier approach in the spirit of a von Neumann stability analysis. Remarkably, the Fourier analysis provides a deep understanding of the stability of the non-linear DLRA evolution equations, despite being a tool for linear problems. This mainly stems from the fact that non-linearities only arise in the basis functions, which, by Parselval's identity, do not affect the $L^2$-norm of the solution. The analysis recovers the behaviour seen in numerical experiments and enhances the understanding of dampening effects that are observed in the different DLRA approaches. To the best of our knowledge, a stability estimate of numerical schemes for DLRA evolution equations is only available in the case of the matrix projector-splitting integrator applied to uncertain parabolic problems \cite{kazashi2020stability}. This analysis uses a discrete variational principle, which does not apply for hyperbolic and kinetic problems investigated in this work. The Fourier approach chosen in our work enables us to propose a stable discretization of the continuous projector splitting based dynamical low-rank approximation. In contrast to previously derived schemes, the resulting discretization for the projector-splitting integrator is $L^2$-stable in the stability region of the full problem.  Furthermore, we introduce a stable and efficient discretization of scattering for radiation transport. For this, we split scattering and streaming parts which is a common practice in radiation transport \cite{adams2002fast,hauck2013collision,crockatt2017arbitrary}. By noting that the integrator for the scattering part only imposes dynamics in the $L$-step, we can omit the remainder, which reduces computational costs and provides a stable treatment of the scattering terms.

This paper is structured as follows: After this introduction, we provide a general background to the used methods in Section~\ref{sec:backgrounds} to give an overview on existing work and to fix notation. Here, we derive a spatial discretization of the original problem in Section~\ref{sec:discretizationFull}, then we present a stability analysis of this discretization in Section~\ref{sec:L2StabilitySchemeFull} and include scattering into this scheme in Section~\ref{sec:backgroundScattering}. A short review of dynamical low-rank approximation is provided in Section~\ref{sec:backgroundDLR}, with a focus on the two robust integrators as well as their discrete and continuous formulations. Section~\ref{sec:L2StabilityProjectorSplitting} presents the stability analysis for the matrix projector-splitting integrator. We start by pointing out potential stability issues for the discrete DLRA formulation in Section~\ref{sec:L2StabilityDLRAdiscrete}, and propose a stable discretization for the continuous formulation in Section~\ref{sec:lowRankThenDiscretize}. In Section~\ref{sec:L2StabilityDLRAdiscreteUnc}, the proposed stability analysis is applied to the unconventional integrator which is shown to be stable even for the discrete DLRA formulation. An efficient and stable treatment of scattering terms that arise in kinetic transport is discussed in Section~\ref{sec:scattering} and we provide numerical examples in Section~\ref{sec:results}.

\section{Background}\label{sec:backgrounds}
\subsection{Discretization of the full problem}\label{sec:discretizationFull}
Parametric linear systems play an important role in various applications such as radiative transport or uncertainty quantification for material deformations. In the following, let us study a linear system of the form
\begin{align}\label{eq:advectionSystem}
\partial_t \mathbf u(t,x) = -\mathbf{A}\partial_x\mathbf u(t,x),
\end{align}
where we have $\mathbf u = (u_1,\cdots,u_{N+1})\in\mathbb{R}^N$ and $\mathbf A = (a_{\ell k})_{\ell,k=1}^{N+1}\in\mathbb{R}^{(N+1)\times (N+1)}$. Such systems can for example arise in the P$_N$ or S$_N$ approximations to radiative transfer \cite{case1967linear,lewis1984computational,adams2002fast} as well as stochastic-Galerkin approximations for linear problems with uncertainty \cite{gottlieb2008galerkin,gerster2019discretized}. Let us discretize the above equation in space using a finite volume approximation with Lax-Friedrichs numerical flux. The spatial domain is decomposed into grid cells $I_j = [x_j,x_{j+1}]$ with equidistant spacing $\Delta x$. A semi-discrete method for the solution $\mathbf u_{j} = (u_{jk})_{k=1}^{N+1}$, where $u_{jk} := \frac{1}{\Delta x}\int_{I_j}u_k(t,x)\,dx$ then takes the form
\begin{align*}
    \dot{\mathbf u}_{j}(t) = -\frac{1}{\Delta x}\left(\mathbf{f}^*(\mathbf u_{j}(t),\mathbf u_{j+1}(t))-\mathbf{f}^*(\mathbf u_{j-1}(t),\mathbf u_{j}(t))\right).
\end{align*}
The Lax-Friedrichs numerical flux with input $\mathbf u,\mathbf v\in\mathbb{R}^{N+1}$ reads
\begin{align*}
    \mathbf{f}^*(\mathbf u,\mathbf v) = \frac12\left(\mathbf A (\mathbf u+\mathbf v) -\frac{\Delta x}{\Delta t}(\mathbf v-\mathbf u)\right).
\end{align*}
Writing the scheme without the definition of numerical fluxes gives
\begin{align}\label{eq:PNSemiDiscrete}
\dot{\mathbf u}_{j}(t)= \frac{\mathbf u_{j-1}(t)-2\mathbf u_{j}(t)+\mathbf u_{j+1}(t)}{2\Delta t}- \frac{1}{2\Delta x}\mathbf A(\mathbf u_{j+1}(t)-\mathbf u_{j-1}(t)).
\end{align}
To simplify notation, we rewrite the time update in matrix notation. Let us define the tridiagonal matrices $\mathbf L^{(1)},\mathbf L^{(2)}\in\mathbb{R}^{N_x\times N_x}$ with non-zero entries in the off-diagonals
\begin{align*}
L^{(1)}_{j,j+1} = L^{(1)}_{j,j-1} = \frac{1}{2}, \quad\text{and }\enskip L^{(2)}_{j,j\pm1} = \pm\frac{\Delta t}{2\Delta x}.
\end{align*}
Then, when collecting the solution in $\mathbf u = (\mathbf u_{j})_{j= 1}^{N_x}\in\mathbb{R}^{N_x\times (N+1)}$, the scheme \eqref{eq:PNSemiDiscrete} becomes
\begin{align}\label{eq:schemeMatrixNotationSemiDiscrete}
\dot{\mathbf{u}}(t) = \frac{1}{\Delta t}\left((\mathbf L^{(1)}-\mathbf I)\mathbf u(t)-\mathbf L^{(2)}\mathbf u(t)\mathbf{A}^T\right) =: \mathbf{F}(\mathbf{u}(t)).
\end{align}
Using a forward Euler time discretization with time step size $\Delta t$ and using $\mathbf u^{n} \approx \mathbf u(t_n)$, gives the fully discrete scheme
\begin{align}\label{eq:schemeMatrixNotation}
\mathbf{u}^{n+1} = \mathbf L^{(1)}\mathbf u^{n}-\mathbf L^{(2)}\mathbf u^{n}\mathbf{A}^T.
\end{align}
\subsection{$L^2$-stability analysis for the full problem}\label{sec:L2StabilitySchemeFull}
To recall certain details in the classical $L^2$-stability analysis and to fix notation, let us start by recalling the $L^2$-stability analysis for the full problem. Without loss of generality, we assume the spatial domain to be the interval $[-1,1]$. In this case, a discrete Fourier ansatz for the discretized solution takes the form
\begin{align}\label{eq:WaveAnsatzSystem}
u_{jk}(t) =: u_{k}(t,x_j)  = \sqrt{\frac{\Delta x}{N+1}}\sum_{\alpha=1}^{N_x}\sum_{\ell=1}^{N+1}\hat u_{\alpha\ell}(t)\exp(i\alpha\pi x_j)\exp\left(2\pi i\frac{\ell k}{N+1}\right).
\end{align}
Here, $i$ denotes the imaginary unit. Collecting the basis functions in the matrices 
\begin{align*}
\mathbf{E}_x = \left(\sqrt{\Delta x}\exp(i\alpha\pi x_j)\right)_{j,\alpha=1}^{N_x}\text{ and }\mathbf{E}_{\mu} = \left(\frac{1}{\sqrt{N+1}}\exp\left(2\pi i\frac{\ell m}{N+1}\right)\right)_{m,\ell=1}^{N+1}
\end{align*}
lets us write the Fourier ansatz \eqref{eq:WaveAnsatzSystem} at time $t_n$ in matrix notation as $\mathbf{u}^n = \mathbf{E}_x\mathbf{\hat{u}}^n\mathbf{E}_{\mu}^H$. We use an upper case $H$ to indicate the adjoint matrix. The next step is to plug this wave ansatz into \eqref{eq:schemeMatrixNotation}, which gives
\begin{align*}
\mathbf u^{n+1} = \mathbf L^{(1)}\mathbf{E}_x\mathbf{\hat u}^n\mathbf{E}_{\mu}^H-\mathbf L^{(2)}\mathbf{E}_x\mathbf{\hat u}^n\mathbf{E}_{\mu}^H\mathbf{A}^T.
\end{align*}
The choice of our ansatz will simplify this scheme, since
\begin{align}\label{eq:FourierResponseSpatialScheme}
\mathbf L^{(1)}\mathbf{E}_x = \mathbf{E}_x \mathbf{D}^{(1)} , \quad\text{and }\enskip \mathbf L^{(2)}\mathbf{E}_x = \mathbf{E}_x \mathbf{D}^{(2)}.
\end{align}
Here, the diagonal matrices $\mathbf{D}^{(1)},\mathbf{D}^{(2)}\in\mathbb{C}^{N_x\times N_x}$ have entries
\begin{align*}
    D_{\alpha\beta}^{(1)} = \frac{e^{i\alpha\pi \Delta x}+e^{-i\alpha\pi \Delta x}}{2}\delta_{\alpha\beta} = cos(\alpha \pi \Delta x)\delta_{\alpha \beta},\quad D_{\alpha\beta}^{(2)} = \frac{\Delta t}{2\Delta x}(e^{i\alpha\pi \Delta x}-e^{-i\alpha\pi \Delta x})\delta_{\alpha\beta}= \frac{i \Delta t}{\Delta x} \sin(\alpha \pi \Delta x).
\end{align*}
Hence, the Fourier ansatz simplifies the scheme to
\begin{align*}
\mathbf u^{n+1} =&\left(\mathbf{E}_x\mathbf{D}^{(1)}\mathbf{\hat u}^n\mathbf{E}_{\mu}^H-\mathbf{E}_x\mathbf{D}^{(2)}\mathbf{\hat u}^n\mathbf{E}_{\mu}^H\mathbf{A}^T\right)\\
=& \mathbf{E}_x\left(\mathbf{D}^{(1)}\mathbf{\hat u}^n-\mathbf{D}^{(2)}\mathbf{\hat u}^n\mathbf{E}_{\mu}^H\mathbf{A}^T\mathbf{E}_{\mu}\right)\mathbf{E}_{\mu}^H.
\end{align*}
Now, with $\mathbf u^{n+1} = \mathbf{E}_x\mathbf{\hat u}^{n+1}\mathbf{E}_{\mu}^H$, we directly see that
\begin{align*}
\mathbf{\hat u}^{n+1} = \mathbf{D}^{(1)}\mathbf{\hat u}^n-\mathbf{D}^{(2)}\mathbf{\hat u}^n\mathbf{E}_{\mu}^H\mathbf{A}^T\mathbf{E}_{\mu}.
\end{align*}
We are interested in deriving an estimate for the Frobenius norm of $\mathbf{\hat u}^{n+1}$ (i.e.~the $L^2$ norm of the vector containing all degrees of freedom), which we denote by $\Vert\mathbf{\hat u}^{n+1}\Vert_F$. In the following we use $\Vert \cdot \Vert$ to denote the spectral matrix norm and $\Vert\cdot\Vert_2$ to denote the Euclidean norm for vectors. 

Collecting the Fourier coefficients in a vector $\mathbf{\hat u}_{\alpha}^{n}=\left(\hat u_{\alpha,\ell}^{n}\right)_{\ell=1}^{N+1}$ gives the time update
\begin{align*}
\mathbf{\hat u}_{\alpha}^{n+1,T} =& \mathbf{\hat u}_{\alpha}^{n,T}\bigg(\cos(\alpha \pi \Delta x)\mathbf{I}-  \frac{i \Delta t}{\Delta x} \sin(\alpha \pi \Delta x)\mathbf{E}_{\mu}^H\mathbf A^T\mathbf{E}_{\mu}\bigg).
\end{align*}
Hence, when denoting the $k_{th}$ eigenvalue of $\mathbf A$ as $\lambda_{k}(\mathbf A)$, the Euclidean norm gives for every $\alpha$
\begin{align*}
\Vert \mathbf{\hat u}_{\alpha}^{n+1}\Vert_2 \leq&  \max_{k}\left|  \cos(\alpha \pi \Delta x)  - \frac{i \Delta t}{\Delta x} \sin(\alpha \pi \Delta x)\lambda_{k}(\mathbf A)\right|\Vert \mathbf{\hat u}^n_{\alpha}\Vert_2\\
=& \max_{k}\left|D^{(1)}_{\alpha\alpha}-D^{(2)}_{\alpha\alpha}\lambda_{k}(\mathbf A)\right|\Vert \mathbf{\hat u}^n_{\alpha}\Vert_2.
\end{align*}
We thus have that
\begin{align*}
    \left|D^{(1)}_{\alpha\alpha}-D^{(2)}_{\alpha\alpha}\lambda_{k}(\mathbf A)\right| = \sqrt{\cos^2(\alpha\pi\Delta x)+\lambda_{k}(\mathbf A)^2\frac{\Delta t^2}{\Delta x^2}\sin^2(\alpha\pi\Delta x)}.
\end{align*}
Hence, the eigenvalue which maximizes the amplification is $\lambda_{max}(\mathbf A)$, which denotes the biggest absolute eigenvalue of $\mathbf A$. Then, the amplification of a Fourier mode with wave number $\alpha$ becomes
\begin{align}\label{eq:componentUpdate}
\Vert \mathbf{\hat u}_{\alpha}^{n+1}\Vert_2 \leq \left|D^{(1)}_{\alpha\alpha}-D^{(2)}_{\alpha\alpha}\lambda_{max}(\mathbf A)\right|\Vert \mathbf{\hat u}^n_{\alpha}\Vert_2.
\end{align}
Let us store the amplification factor in a diagonal matrix $\mathbf{D}\in\mathbb{R}^{N_x\times N_x}$ with
\begin{align*}
    D_{\alpha\alpha} = \left|D^{(1)}_{\alpha\alpha}-D^{(2)}_{\alpha\alpha}\lambda_{max}(\mathbf A)\right|
\end{align*}
and collect the norm at wave number $\alpha$ in a vector $\mathbf e^n = \left(\Vert \mathbf{\hat u}_{\alpha}^{n}\Vert_2\right)_{\alpha = 1}^{N_x}$. Due to \eqref{eq:componentUpdate}, the estimate $\mathbf{e}^{n+1}\leq \mathbf{D}\mathbf{e}^{n}$ holds component-wise. Therefore, we have
\begin{align*}
\Vert \mathbf{e}^{n+1} \Vert_2 \leq  \Vert \mathbf{D}\Vert\cdot\Vert\mathbf{e}^{n}\Vert_2 = \lambda_{max}(\mathbf{D})\Vert\mathbf{e}^{n}\Vert_2.
\end{align*}
Hence, for the Frobenius norm, we obtain
\begin{align*}
\Vert \mathbf{\hat u}^{n+1} \Vert_{F} \leq \max_{\alpha} \vert D_{\alpha\alpha}\vert\cdot \Vert \mathbf{\hat u}^{n} \Vert_{F}.
\end{align*}
Due to Parseval's identity, we have
\begin{align}\label{eq:FrobeniusNormEstimate}
\Vert \mathbf{ u}^{n+1} \Vert_{F} =\left\Vert \mathbf{E}_x \mathbf{\hat u}^{n+1} \mathbf{E}_{\mu}^H\right\Vert_F\leq \max_{\alpha}\vert D_{\alpha\alpha}\vert\cdot \Vert \mathbf{ u}^{n} \Vert_{F}.
\end{align}
When using the CFL number $c = \lambda_{max}(\mathbf A)\Delta t/\Delta x$ we obtain
\begin{align*}
\left\vert D_{\alpha\alpha}\right\vert \leq\sqrt{\cos^2(\alpha\pi\Delta x)+c^2\sin^2(\alpha\pi\Delta x)}.
\end{align*}
To obtain $L^2$-stability we require an amplification factor which is smaller or equal to one, i.e., we must pick $c\leq 1$. For linear schemes, this stability (together with consistency) can be used to prove convergence. However, in this work, we focus on understanding dampening properties of the different integrators and leave the question of convergence for the (necessarily nonlinear) dynamical low-rank approximation to future research.
\subsection{Stability for scattering terms}\label{sec:backgroundScattering}
In the following, let us focus on the application of radiative transport. In this case, the original advection system \eqref{eq:advectionSystem} is augmented by scattering and absorption effects. This leads to the P$_N$ equations, which read
\begin{align}\label{eq:PN}
\partial_t \mathbf u(t,x) = -\mathbf A\partial_x \mathbf u(t,x)-\sigma_a \mathbf u(t,x) + \sigma_s
\mathcal{S}\mathbf u(t,x) \quad \text{ with }a_{km}=\int_{-1}^1 \mu P_k(\mu) P_m(\mu)\,d\mu.
\end{align}
Here, $P_k$ denotes the Legendre polynomial of order $k$ and $\mathcal{S}\in\mathbb{R}^{(N+1)\times(N+1)}$ is a scattering matrix. The variable $\mu\in[-1,1]$ is the projected direction in which particles travel. To shorten notation, we define $\mathbf G := \sigma_s\mathcal{S}-\sigma_a\mathbf I$ with entries $G_{kk} = \sigma_s g_k -\sigma_a$. For isotropic scattering one for example has $g_k = \delta_{k0}-1$, i.e. scattering will not directly affect the scalar flux while dampening higher order moments. Let us investigate how the additional scattering affects stability. Commonly, scattering and streaming are treated separately through a splitting step, see e.g. \cite{adams2002fast}. In this case, we can update the solution from time $t_0$ to time $t_1$ by
\begin{subequations}\label{eq:streamingScatteringEqns}
\begin{alignat}{2}
\partial_t\mathbf{u}_{\RomanNumeralCaps{1}}(t,x) &= -\mathbf A\partial_x \mathbf u_{\RomanNumeralCaps{1}}(t,x), \qquad&& \mathbf{u}_{\RomanNumeralCaps{1}}(t_0,x) = \mathbf{u}(t_0,x)\label{eq:schemeMatrixNotationA2cont}\\
\partial_t\mathbf{u}_{\RomanNumeralCaps{2}}(t,x) &= \mathbf G\mathbf u_{\RomanNumeralCaps{2}}(t,x),\qquad&& \mathbf{u}_{\RomanNumeralCaps{2}}(t_0,x) = \mathbf{u}_{\RomanNumeralCaps{1}}(t_1,x)\label{eq:schemeMatrixNotationB2cont}.
\end{alignat}
\end{subequations}
Choosing the discretization proposed in Section~\ref{sec:discretizationFull}, the update of the full problem is composed of the two substeps
\begin{subequations}
\begin{align}
\mathbf{u}^{n+1/2} &= \mathbf L^{(1)}\mathbf u^{n}-\mathbf L^{(2)}\mathbf u^{n}\mathbf{A}^T,\label{eq:schemeMatrixNotationA2}\\
\mathbf{u}^{n+1} &= \mathbf{u}^{n+1/2}+\Delta t\mathbf u^{n+1/2}\mathbf G\label{eq:schemeMatrixNotationB2}.
\end{align}
\end{subequations}
Written more compactly as a single update, the scheme becomes
\begin{align}\label{eq:schemeMatrixNotationNew}
\mathbf{u}^{n+1} = \left(\mathbf L^{(1)}\mathbf u^{n}-\mathbf L^{(2)}\mathbf u^{n}\mathbf{A}^T\right)\left(\mathbf I + \Delta t \mathbf G\right)
\end{align}
Let us again use a discrete Fourier ansatz $\mathbf u^n = \mathbf{E}_x\mathbf{\hat u}^n\mathbf{E}_{\mu}^H$. The next step is to plug this wave ansatz into \eqref{eq:schemeMatrixNotationNew}, which gives
\begin{align*}
\mathbf u^{n+1} =&\left(\mathbf{E}_x\mathbf{D}^{(1)}\mathbf{\hat u}^n\mathbf{E}_{\mu}^H-\mathbf{E}_x\mathbf{D}^{(2)}\mathbf{\hat u}^n\mathbf{E}_{\mu}^H\mathbf{A}^T\right)\left( \mathbf I +\Delta t\mathbf G\right)\\
=& \left(\mathbf{E}_x\mathbf{D}^{(1)}\mathbf{\hat u}^n\mathbf{E}_{\mu}^H-\mathbf{E}_x\mathbf{D}^{(2)}\mathbf{\hat u}^n\mathbf{E}_{\mu}^H\mathbf{A}^T\mathbf{E}_{\mu}\mathbf{E}_{\mu}^H\right)\left( \mathbf{E}_{\mu}\mathbf{E}_{\mu}^H +\Delta t\mathbf G\mathbf{E}_{\mu}\mathbf{E}_{\mu}^H\right)\\
=& \mathbf{E}_x\left(\mathbf{D}^{(1)}\mathbf{\hat u}^n-\mathbf{D}^{(2)}\mathbf{\hat u}^n\mathbf{E}_{\mu}^H\mathbf{A}^T\mathbf{E}_{\mu}\right)\left( \mathbf{I} +\Delta t\mathbf{E}_{\mu}^H\mathbf G\mathbf{E}_{\mu}\right)\mathbf{E}_{\mu}^H.
\end{align*}
Now, with $\mathbf u^{n+1} = \mathbf{E}_x\mathbf{\hat u}^{n+1}\mathbf{E}_{\mu}^H$, we directly see that
\begin{align*}
\mathbf{\hat u}^{n+1} = \left(\mathbf{D}^{(1)}\mathbf{\hat u}^n-\mathbf{D}^{(2)}\mathbf{\hat u}^n\mathbf{E}_{\mu}^H\mathbf{A}^T\mathbf{E}_{\mu}\right)\left( \mathbf{I} +\Delta t\mathbf{E}_{\mu}^H\mathbf G\mathbf{E}_{\mu}\right).
\end{align*}
Using Parseval's identity yields 
\begin{align}\label{eq:FrobeniusNormEstimate}
\Vert \mathbf{ u}^{n+1} \Vert_F =\left\Vert \left( \mathbf{D}^{(1)}\mathbf{\hat u}^n-\mathbf{D}^{(2)}\mathbf{\hat u}^n\mathbf{E}_{\mu}^H\mathbf{A}^T\mathbf{E}_{\mu}\right)\left( \mathbf{I} +\Delta t\mathbf{E}_{\mu}^H\mathbf G\mathbf{E}_{\mu}\right) \right\Vert_F\leq \max_{\alpha}\vert \widetilde D_{\alpha\alpha}\vert\cdot \Vert \mathbf{ u}^{n} \Vert_F
\end{align} 
with
\begin{align*}
    \widetilde D_{\alpha\alpha} = \left|D^{(1)}_{\alpha\alpha}-D^{(2)}_{\alpha\alpha}\lambda_{max}(\mathbf A)\right|\cdot \max_{\ell}\left|1+\Delta t G_{\ell\ell}\right|.
\end{align*}
With the CFL number $\lambda_{max}(\mathbf A)\Delta t/\Delta x = c$, this gives
\begin{align*}
\Vert \mathbf{ u}^{n+1} \Vert_F \leq \max_{\ell}\left|1+\Delta t G_{\ell\ell}\right|\cdot\sqrt{\cos^2(\alpha\pi\Delta x)+c^2\sin^2(\alpha\pi\Delta x)}\cdot\Vert \mathbf{ u}^{n} \Vert_F.
\end{align*}

\subsection{Dynamical low-rank approximation}\label{sec:backgroundDLR}
In the following, we give a short overview on dynamical low-rank approximation \cite{KochLubich07} for problems of the form \eqref{eq:advectionSystem}. The main idea of DLRA is to represent and evolve the solution on a manifold of rank $r$ functions. There are two approaches to derive the evolution equations of dynamical low-rank approximation. The first one chooses a low-rank approximation on the matrix solution of \eqref{eq:schemeMatrixNotationSemiDiscrete} and the second one chooses a low-rank approximation on the continuous level for the solution of the original problem \eqref{eq:advectionSystem}, which is subsequently discretized. 

Let us start by presenting DLRA for the discrete system \eqref{eq:schemeMatrixNotationSemiDiscrete}. In this case, the solution $\mathbf u(t)\in\mathbb{R}^{N_x\times(N+1)}$ is represented by
\begin{align}\label{eq:rankrsol}
\mathbf u(t)\approx\mathbf{X}(t)\mathbf{S}(t)\mathbf{W}(t)^T,
\end{align}
where $\mathbf{X}\in\mathbb{R}^{N_x\times r}$, $\mathbf{S}\in\mathbb{R}^{r\times r}$ and $\mathbf{W}\in\mathbb{R}^{(N+1)\times r}$. The aim is to derive evolution equations for each of these factorization matrices. Let us denote the set of matrices that have the form \eqref{eq:rankrsol} by $\mathcal{M}_r$. Then, we wish to find $\mathbf u_r\in\mathcal{M}_r$ which fulfills
\begin{align}\label{eq:DLRproblem}
\dot{\mathbf u}_r(t)\in T_{ \mathbf u_r(t)}\mathcal{M}_r \qquad \text{such that} \qquad \left\Vert \dot{\mathbf u}_r(t)-\mathbf{F}(\mathbf u(t)) \right\Vert = \text{min},
\end{align}
where $\mathbf F$ denotes the right-hand side of the semi-discrete scheme \eqref{eq:schemeMatrixNotationSemiDiscrete}. We use $T_{\mathbf u_r(t)}\mathcal{M}_r$ to denote the tangent space of $\mathcal{M}_r$ at $\mathbf u_r(t)$. The stated problem can be reformulated \cite[Lemma~4.1]{KochLubich07} as 
\begin{align}\label{eq:lowRankProjector}
\dot{\mathbf u}_r(t) = \mathbf{P}(\mathbf u_r(t))\mathbf{F}(\mathbf u_r(t)),
\end{align}
where $\mathbf{P}$ denotes the orthogonal projection onto the tangent space, which is given by
\begin{align*}
\mathbf P\mathbf g = \mathbf{X}\mathbf{X}^T \mathbf g - \mathbf{X}\mathbf{X}^T \mathbf g \mathbf{W}\mathbf{W}^T + \mathbf g \mathbf{W}\mathbf{W}^T.
\end{align*}
The evolution equation \eqref{eq:lowRankProjector} is then split by a Lie-Trotter splitting technique, yielding
\begin{subequations}\label{eq:projectorSplitEq}
\begin{alignat}{2}
\dot{\mathbf u}_{\RomanNumeralCaps{1}}(t) &= \mathbf{F}(\mathbf u_{\RomanNumeralCaps{1}}(t))\mathbf{W}\mathbf{W}^T, \quad && \mathbf u_{\RomanNumeralCaps{1}}(t_0) = \mathbf u_r(t_0), \label{eq:DLR1}\\ 
\dot{\mathbf u}_{\RomanNumeralCaps{2}}(t) &= -\mathbf{X}\mathbf{X}^T\mathbf{F}(\mathbf u_{\RomanNumeralCaps{2}}(t))\mathbf{W}\mathbf{W}^T, \quad && \mathbf u_{\RomanNumeralCaps{2}}(t_0) = \mathbf u_{\RomanNumeralCaps{1}}(t_1),\label{eq:DLR2}\\ 
\dot{\mathbf u}_{\RomanNumeralCaps{3}}(t) &= \mathbf{X}\mathbf{X}^T \mathbf{F}(\mathbf u_{\RomanNumeralCaps{3}}(t)) , \quad && \mathbf u_{\RomanNumeralCaps{3}}(t_0) = \mathbf u_{\RomanNumeralCaps{2}}(t_1).\label{eq:DLR3}
\end{alignat}
\end{subequations}
This scheme can be used to update the solution from $\mathbf u_r(t_0)$ to $\mathbf u_r(t_1) = \mathbf u_{\RomanNumeralCaps{3}}(t_1)$. These split equations are reformulated to yield an efficient and robust integrator. Each substep in the above equations has a decomposition of the form \eqref{eq:rankrsol}. Defining the decompositions $\mathbf u_{\RomanNumeralCaps{1}} = \mathbf{K}\mathbf{W}^T$ and $\mathbf u_{\RomanNumeralCaps{3}} = \mathbf{X}\mathbf{L}$ gives the \textit{matrix projector-splitting integrator}
\begin{enumerate}
    \item \textbf{$K$-step}: Update $\mathbf X^{0}$ to $\mathbf X^{1}$ and $\mathbf S^0$ to $\mathbf{\widetilde S}^0$ via
\begin{align}
\dot{\mathbf K}(t) &= \mathbf{F}(\mathbf{K}(t)\mathbf{W}^{0,T})\mathbf{W}^0, \qquad \mathbf K(t_0) = \mathbf{X}^0\mathbf{S}^0.\label{eq:KStepSemiDiscrete}
\end{align}
Determine $\mathbf X^1$ and $\mathbf{\widetilde S}^0$ with $\mathbf K(t_1) = \mathbf X^1 \mathbf{\widetilde S}^0$ by performing a QR decomposition.
\item \textbf{$S$-step}: Update $\mathbf{\widetilde S^0}$ to $\mathbf{\widetilde S^1}$ via
\begin{align}
\dot{\mathbf{\widetilde S}}(t) = -\mathbf{X}^{1,T}\mathbf{F}(\mathbf{X}^{1}\mathbf{\widetilde S}(t)\mathbf{W}^{0,T})\mathbf{W}^0, \qquad \mathbf{\widetilde S}(t_0) = \mathbf{\widetilde S}^0\label{eq:SStepSemiDiscrete}
\end{align}
and set $\mathbf{\widetilde S}^1 = \mathbf{\widetilde S}(t_1)$.
\item \textbf{$L$-step}: Update $\mathbf W^0$ to $\mathbf W^1$ and $\mathbf{\widetilde S}^1$ to $\mathbf S^1$ via
\begin{align}
\dot{\mathbf L}(t) &= \mathbf{X}^{1,T}\mathbf{F}(\mathbf{X}^1\mathbf{L}(t)), \qquad \mathbf L(t_0) = \mathbf{\widetilde S}^1 \mathbf{W}^{0,T}.\label{eq:LStepSemiDiscrete}
\end{align}
Determine $\mathbf W^1$ and $\mathbf S^1$ with $\mathbf L(t_1) = \mathbf S^1 \mathbf W^{1,T}$ by performing a QR decomposition.
\end{enumerate}
The time updated solution is then given by $\mathbf{u}_r(t_1) = \mathbf{X}^1\mathbf{S}^1\mathbf{W}^{1,T}$. For more details on the matrix projector-splitting integrator, we refer to \cite{LubichOseledets}. 

Recently, a further robust integrator, called the \textit{unconventional integrator}, has been introduced in \cite{CeL21}. This integrator works as follows:
\begin{enumerate}
    \item \textbf{$K$-step}: Update $\mathbf X^{0}$ to $\mathbf X^{1}$ via
    \begin{align}
        \dot{\mathbf K}(t) &= \mathbf{F}(\mathbf{K}(t)\mathbf{W}^{0,T})\mathbf{W}^0, \qquad \mathbf K(t_0) = \mathbf{X}^0\mathbf{S}^0.\label{eq:KStepSemiDiscreteUI}
    \end{align}
Determine $\mathbf X^1$ with $\mathbf K(t_1) = \mathbf X^1 \mathbf R$ and store $\mathbf M = \mathbf X^{1,T}\mathbf X^0$.
\item \textbf{$L$-step}: Update $\mathbf W^0$ to $\mathbf W^1$ via
\begin{align}
\dot{\mathbf L}(t) &= \mathbf{X}^{0,T}\mathbf{F}(\mathbf{X}^0\mathbf{L}(t)), \qquad \mathbf L(t_0) = \mathbf{S}^0 \mathbf{W}^{0,T}.\label{eq:LStepSemiDiscreteUI}
\end{align}
Determine $\mathbf W^1$ with $\mathbf L^1 = \mathbf W^1\mathbf{\widetilde R}$ and store $\mathbf N = \mathbf W^{1,T} \mathbf W^0$.
\item \textbf{$S$-step}: Update $\mathbf S^0$ to $\mathbf S^1$ via
\begin{align}
\dot{\mathbf S}(t) = \mathbf{X}^{1,T}\mathbf{F}(\mathbf{X}^{1}\mathbf{S}(t)\mathbf{W}^{1,T})\mathbf{W}^1, \qquad \mathbf S(t_0) &= \mathbf M\mathbf S^0 \mathbf N^T\label{eq:SStepSemiDiscreteUI}
\end{align}
and set $\mathbf S^1 = \mathbf S(t_1)$.
\end{enumerate}
Note that these two integrators take the semi-discrete matrix ODE system \eqref{eq:schemeMatrixNotationSemiDiscrete} as a starting point to derive DLRA evolution equations. I.e., the evolution equations are derived after performing the spatial discretization. Following \cite{EiL18}, the DLRA evolution equations can also be derived for the continuous problem first, and the spatial discretization is performed on the DLRA equations second. Note that in our case, we are starting from a large system of partial differential equations \eqref{eq:advectionSystem}, which can result from a discretization of the directional or uncertain domain of transport equations or linear equations with uncertainty. In our analysis, only the discretization of the spatial domain is important, which is why it does not matter whether the original problem is the P$_N$ (as well as stochastic-Galerkin) system or the scalar transport equation (or uncertain linear equation). Starting at a system of the form $\eqref{eq:advectionSystem}$, the low-rank solution ansatz is
\begin{align}\label{eq:rankrsolcont}
\mathbf u_r(t,x)=\sum_{j,\ell=1}^r X_j(t,x)S_{j\ell}(t)\mathbf W_{\ell}(t).
\end{align}
Note that we now have basis functions $X_j:\mathbb{R}_+\times \mathbb{R}\rightarrow \mathbb{R}$ and $\mathbf W_{\ell}:\mathbb{R}_+\rightarrow \mathbb{R}^{N+1}$. The corresponding split equations \eqref{eq:projectorSplitEq} become
\begin{subequations}\label{eq:projSplittingContProblem}
\begin{alignat}{2}
\partial_t\mathbf u_{\RomanNumeralCaps{1}}(t,x) =& -\left(\mathbf{A}\partial_x \mathbf u_{\RomanNumeralCaps{1}}(t,x)\right)\mathbf{W}\mathbf{W}^T \quad && \mathbf u_{\RomanNumeralCaps{1}}(t_0,x) = \mathbf u_r(t_0,x),\\
\partial_t\mathbf u_{\RomanNumeralCaps{2}}(t,x) =&P_{X}\left(\mathbf{A}\partial_x \mathbf u_{\RomanNumeralCaps{2}}(t,x)\right)\mathbf{W}\mathbf{W}^T, \quad && \mathbf u_{\RomanNumeralCaps{2}}(t_0,x) = \mathbf u_{\RomanNumeralCaps{1}}(t_1,x),\\
\partial_t\mathbf u_{\RomanNumeralCaps{3}}(t,x) =& -P_{X}\left(\mathbf{A}\partial_x \mathbf u_{\RomanNumeralCaps{3}}(t,x)\right), \quad && \mathbf u_{\RomanNumeralCaps{3}}(t_0,x) = \mathbf u_{\RomanNumeralCaps{2}}(t_1,x).
\end{alignat}
\end{subequations}
where we have $P_{X}g := \sum_{i=1}^{r} \langle g, X_i(t,\cdot)\rangle X_i$ and we choose $\langle\cdot,\cdot\rangle$ to denote the $L^2$ inner product with respect to space. Furthermore, we will use the notation $\langle \cdot\rangle$ to indicate an integration over the spatial domain. Then, when collecting the spatial basis functions in the vector $\mathbf{X}(t,x) = (X_j(t,x))_{j=1}^r$ and storing the vectors $\mathbf{W}_{\ell}$ as columns of the matrix $\mathbf{W}\in\mathbb{R}^{(N+1)\times r}$, the corresponding $K$, $S$ and $L$-equations read
\begin{subequations}\label{eq:projSplittingContProblemKL}
\begin{align}
\partial_t \mathbf K(t,x) =& -\mathbf{W}^T\mathbf{A}\mathbf W\partial_x \mathbf K(t,x),\\
\dot{\mathbf S}(t) =& \mathbf{W}^T\mathbf{A}\mathbf W\mathbf S(t)^T \langle\partial_x \mathbf X \mathbf X^{T}\rangle  ,\\
\dot{\mathbf L}(t) =& -\mathbf{A}\mathbf L(t)\langle\partial_x \mathbf X\mathbf X^T\rangle .
\end{align}
\end{subequations}
Note that we use dots to indicate the time derivative for ordinary differential equations, whereas a partial time derivative is used for partial differential equations. The continuous formulation of the unconventional integrator takes the same $K$, $S$ and $L$ steps, but uses a different ordering. Note that the formulation \eqref{eq:projSplittingContProblemKL} is continuous in space and requires a spatial discretization in order to evolve the system numerically in time. Compared to deriving the DLRA equations for the disrcete matrix ODE, this formulation provides more freedom in the choice of discretizations of each individual equation. At the same time, choosing such a discretization requires a profound understanding of the stability related to this set of equations. 

\section{$L^2$-stability analysis for the matrix projector-splitting integrator}\label{sec:L2StabilityProjectorSplitting}
\subsection{Discrete dynamical low-rank approximation}\label{sec:L2StabilityDLRAdiscretePS}\label{sec:L2StabilityDLRAdiscrete}
In the following, we apply the projector-splitting integrator to the matrix ordinary differential equation \eqref{eq:schemeMatrixNotationSemiDiscrete}. Note that this corresponds to discretizing the full problem first and deriving the DLRA equations second. The $K$, $S$ and $L$ steps from equations \eqref{eq:KStepSemiDiscrete}, \eqref{eq:SStepSemiDiscrete} and \eqref{eq:LStepSemiDiscrete} in combination with an explicit Euler time discretization then read
\begin{subequations}\label{eq:projSplittingMatrixNotationKSL}
\begin{alignat}{2}
\mathbf{K}^{n+1} =& \mathbf{K}^{n} + \left((\mathbf L^{(1)}-\mathbf I)\mathbf u^{n}_{\RomanNumeralCaps{1}}-\mathbf L^{(2)}\mathbf u^{n}_{\RomanNumeralCaps{1}}\mathbf{A}^T\right)\mathbf{W}^n, \quad &&{\mathbf{K}^{n+1} = \mathbf{X}^{n+1}\mathbf{\widetilde S}^{n}},\\
\mathbf{\widetilde S}^{n+1} =& \mathbf{\widetilde S}^{n}- \mathbf{X}^{n+1,T}\left((\mathbf L^{(1)}-\mathbf I)\mathbf u^{n}_{\RomanNumeralCaps{2}}-\mathbf L^{(2)}\mathbf u^{n}_{\RomanNumeralCaps{2}}\mathbf{A}^T\right)\mathbf{W}^n,\label{eq:projSplittingMatrixNotationS}\\
\mathbf{L}^{n+1} =& \mathbf{L}^{n}+ \mathbf{X}^{n+1,T}\left((\mathbf L^{(1)}-\mathbf I)\mathbf u^{n}_{\RomanNumeralCaps{3}}-\mathbf L^{(2)}\mathbf u^{n}_{\RomanNumeralCaps{3}}\mathbf{A}^T\right), \quad &&\mathbf{L}^{n+1} = \mathbf{S}^{n+1}\mathbf{W}^{n+1,T}.
\end{alignat}
\end{subequations}
Here, we make use of the DLRA substeps $\mathbf u^{n}_{\RomanNumeralCaps{1}} = \mathbf X^{n}\mathbf{S}^n\mathbf{W}^{n,T}$, $\mathbf u^{n}_{\RomanNumeralCaps{2}} = \mathbf X^{n+1}\mathbf{\widetilde S}^n\mathbf{W}^{n,T}$ and $\mathbf u^{n}_{\RomanNumeralCaps{3}} = \mathbf X^{n+1}\mathbf{\widetilde S}^{n+1}\mathbf{W}^{n,T}$. To underline similarities of the stability analysis to the full problem (cf. Section~\ref{sec:L2StabilitySchemeFull}), let us go one step back to the corresponding split equations \eqref{eq:projectorSplitEq}, which read
\begin{subequations}\label{eq:projSplittingMatrixNotation}
\begin{align}
\mathbf u_{\RomanNumeralCaps{1}}^{n+1} =& \left(\mathbf L^{(1)}\mathbf u^{n}_{\RomanNumeralCaps{1}}-\mathbf L^{(2)}\mathbf u^{n}_{\RomanNumeralCaps{1}}\mathbf{A}^T\right)\mathbf{W}^n\mathbf{W}^{n,T},\\
\mathbf u_{\RomanNumeralCaps{2}}^{n+1} =&- \mathbf{X}^{n+1}\mathbf{X}^{n+1,T}\left((\mathbf L^{(1)}-2\mathbf I)\mathbf u^{n}_{\RomanNumeralCaps{2}}-\mathbf L^{(2)}\mathbf u^{n}_{\RomanNumeralCaps{2}}\mathbf{A}^T\right)\mathbf{W}^n\mathbf{W}^{n,T},\\
\mathbf u_{\RomanNumeralCaps{3}}^{n+1} =& \mathbf{X}^{n+1}\mathbf{X}^{n+1,T}\left(\mathbf L^{(1)}\mathbf u^{n}_{\RomanNumeralCaps{3}}-\mathbf L^{(2)}\mathbf u^{n}_{\RomanNumeralCaps{3}}\mathbf{A}^T\right).
\end{align}
\end{subequations}
Omitting Roman indices, the solution of every substep in \eqref{eq:projSplittingMatrixNotation} is of the form $\mathbf{u}^n = \mathbf{E}_x\mathbf{\tilde u}^n\mathbf{W}^{n,T}$, where $\mathbf{\tilde u}^n \in\mathbb{R}^{N_x\times r}$. This is easily shown as every substep is of the form $\mathbf u = \mathbf X \mathbf S \mathbf W^T$. We thus have $\mathbf u = \mathbf{E}_x\mathbf{E}_x^H\mathbf{X}\mathbf S\mathbf W^T$. Therefore, one can choose $\mathbf{\tilde u}_{\RomanNumeralCaps{1}}^n = \mathbf{E}_x^H\mathbf{X}^n\mathbf S^n$, $\mathbf{\tilde u}_{\RomanNumeralCaps{2}}^n = \mathbf{E}_x^H\mathbf{X}^{n+1}\mathbf{\widetilde S}^n$ and $\mathbf{\tilde u}_{\RomanNumeralCaps{3}}^n = \mathbf{E}_x^H\mathbf{X}^{n+1}\mathbf{\widetilde S}^{n+1}$. Then, the spatial discretization matrices $\mathbf{L}^{(1)}$ and $\mathbf{L}^{(2)}$ can be Fourier transformed according to \eqref{eq:FourierResponseSpatialScheme}, which gives
\begin{align*}
\mathbf u_{\RomanNumeralCaps{1}}^{n+1} =&\mathbf{E}_x\left(\mathbf D^{(1)}\mathbf{\tilde{u}}^{n}_{\RomanNumeralCaps{1}}\mathbf{W}^{n,T}-\mathbf D^{(2)}\mathbf{\tilde{u}}^n_{\RomanNumeralCaps{1}}\mathbf{W}^{n,T}\mathbf{A}^T\right)\mathbf{W}^n\mathbf{W}^{n,T},\\
\mathbf u_{\RomanNumeralCaps{2}}^{n+1} =& -\mathbf{X}^{n+1}\mathbf{X}^{n+1,T}\mathbf{E}_x\left((\mathbf D^{(1)}-2\mathbf I)\mathbf{\tilde{u}}^{n}_{\RomanNumeralCaps{2}}\mathbf{W}^{n,T}-\mathbf D^{(2)}\mathbf{\tilde{u}}^{n}_{\RomanNumeralCaps{2}}\mathbf{W}^{n,T}\mathbf{A}^T\right)\mathbf{W}^n\mathbf{W}^{n,T},\\
\mathbf u_{\RomanNumeralCaps{3}}^{n+1} =& \mathbf{X}^{n+1}\mathbf{X}^{n+1,T}\mathbf{E}_x\left(\mathbf D^{(1)}\mathbf{\tilde{u}}^{n}_{\RomanNumeralCaps{3}}\mathbf{W}^{n,T}-\mathbf D^{(2)}\mathbf{\tilde{u}}^{n}_{\RomanNumeralCaps{3}}\mathbf{W}^{n,T}\mathbf{A}^T\right).
\end{align*}
If we define $\mathbf{\tilde{A}}:=\mathbf{W}^{n,T}\mathbf A^T\mathbf{W}^n$, this simplifies to
\begin{subequations}\label{eq:projSplittingFourier}
\begin{align}
\mathbf u_{\RomanNumeralCaps{1}}^{n+1} =&\mathbf{E}_x\left(\mathbf D^{(1)}\mathbf{\tilde{u}}^{n}_{\RomanNumeralCaps{1}}-\mathbf D^{(2)}\mathbf{\tilde{u}}^n_{\RomanNumeralCaps{1}}\mathbf{\tilde{A}}\right)\mathbf{W}^{n,T},\label{eq:projSplittingFourier1}\\
\mathbf u_{\RomanNumeralCaps{2}}^{n+1} =& -\mathbf{X}^{n+1}\mathbf{X}^{n+1,T}\mathbf{E}_x\left((\mathbf D^{(1)}-2\mathbf I)\mathbf{\tilde{u}}^{n}_{\RomanNumeralCaps{2}}-\mathbf D^{(2)}\mathbf{\tilde{u}}^{n}_{\RomanNumeralCaps{2}}\mathbf{\tilde{A}}\right)\mathbf{W}^{n,T},\label{eq:projSplittingFourier2}\\
\mathbf u_{\RomanNumeralCaps{3}}^{n+1} =& \mathbf{X}^{n+1}\mathbf{X}^{n+1,T}\mathbf{E}_x\left(\mathbf D^{(1)}\mathbf{\tilde{u}}^{n}_{\RomanNumeralCaps{3}}-\mathbf D^{(2)}\mathbf{\tilde{u}}^{n}_{\RomanNumeralCaps{3}}\mathbf{\tilde{A}}\right)\mathbf{W}^{n,T}.\label{eq:projSplittingFourier3}
\end{align}
\end{subequations}
Now, since we know that the Fourier transform of the projector-splitting integrator takes the form \eqref{eq:projSplittingFourier}, we can now investigate its stability properties.

\begin{theorem}\label{th:L2instability} The application of the projector splitting integrator to the Lax--Friedrichs discretization of \eqref{eq:advectionSystem}, given in equation \eqref{eq:projSplittingMatrixNotation}, is $L^2$-unstable.
\end{theorem}
\begin{proof} Let us pick a single mode solution $u^n_{jk} = \exp(i\bar{\alpha}\pi x_j)w_k$ with $\bar{\alpha}$ such that $\cos(\bar{\alpha}\pi\Delta x)=-1$. In a more compact notation, we define the vector $\mathbf e_x = \left(\exp(i\alpha\pi x_j)\right)_{j=1}^{N_x}$ and with an arbitrary normalized vector $\mathbf{w}\in\mathbb{R}^{N+1}$, we have $\mathbf{u}^n = \mathbf{e}_x\mathbf{w}^T$. Plugging this into the equations \eqref{eq:projSplittingMatrixNotation} yields for the first step
\begin{align*}
    \mathbf u_{\RomanNumeralCaps{1}}^{n+1} = \left(\mathbf L^{(1)}\mathbf{e}_x\mathbf{w}^T-\mathbf L^{(2)}\mathbf{e}_x\mathbf{w}^T\mathbf{A}^T\right)\mathbf{w}\mathbf{w}^T = \mathbf{e}_x\left( D_{\bar{\alpha}\bar{\alpha}}^{(1)}- D_{\bar{\alpha}\bar{\alpha}}^{(2)}\mathbf{w}^T\mathbf{A}^T\mathbf{w}\right)\mathbf{w}^T = -\mathbf{e}_x\mathbf{w}^T.
\end{align*}
Here, we use that for our choice of the wave number we have $D_{\bar{\alpha}\bar{\alpha}}^{(1)} = -1$ and $D_{\bar{\alpha}\bar{\alpha}}^{(2)}=0$. Hence, the basis remains unchanged and only the coefficient changes its sign. Then for the second step, we have
\begin{align*}
    \mathbf u_{\RomanNumeralCaps{2}}^{n+1} = \mathbf{e}_x\mathbf{e}_x^H\left(\mathbf{e}_x(D_{\bar{\alpha}\bar{\alpha}}^{(1)}-2)\mathbf{w}^T-\mathbf{e}_x D_{\bar{\alpha}\bar{\alpha}}^{(2)}\mathbf{w}^T\mathbf{A}^T\right)\mathbf{w}\mathbf{w}^T = -3\mathbf{e}_x\mathbf{w}^T.
\end{align*}
The last step gives 
\begin{align*}
\mathbf u_{\RomanNumeralCaps{3}}^{n+1} = -3\mathbf{e}_x\mathbf{e}_x^H\left(\mathbf{e}_x D_{\bar{\alpha}\bar{\alpha}}^{(1)}\mathbf{w}^T-\mathbf{e}_x D_{\bar{\alpha}\bar{\alpha}}^{(1)}\mathbf{w}^T\mathbf{A}^T\right) = 3\mathbf{e}_x\mathbf{w}^T.
\end{align*}
Hence for this choice of wave number, the Frobenius norm of the solution is amplified by a factor of $3$, i.e., the scheme is not stable. 
\end{proof}

It is clear from the proof of Theorem \ref{th:L2instability} that the $K$ and $L$ step, equations (\ref{eq:projSplittingMatrixNotation}a) and (\ref{eq:projSplittingMatrixNotation}c) respectively, do not amplify the solution. This is in contrast to the $S$ step, equation (\ref{eq:projSplittingMatrixNotation}b). The reason for this is, as we will explain in the subsequent sections, that the $S$ step in the projector splitting integrates backward in time. Thus, the stabilization imposed by the Lax--Friedrich discretization thus acts as an amplification that leads to an unstable scheme (independent of the time step size).

\subsection{Continuous dynamical low-rank approximation}\label{sec:lowRankThenDiscretize}
Previously, we applied DLRA to the discretized system \eqref{eq:schemeMatrixNotation}. Let us now first apply the DLRA method to the spatially continuous problem \eqref{eq:advectionSystem} and then discretizing the resulting differential equations. This continuous approach has been proposed in \cite{EiL18}. Coupled with an appropriate fully implicit scheme it can be shown to be unconditionally stable \cite{ding2019dynamical}. However, here we are interested in an explicit discretization. In this case the approach comes with the freedom to choose stabilization and derivative approximations in each equation individually. In contrast, when discretizing first and applying DLRA second, the stabilization is fixed and inherited by the discretization of the full problem. While the discretization and stabilization of the full problem is well understood, the gained freedom when applying low-rank first requires additional knowledge on the DLRA system which we aim to establish in this section. 

We look at two discretization strategies for the projector-splitting integrator. The projector-splitting integrator applied to the continuous problem \eqref{eq:advectionSystem} leads to the system \eqref{eq:projSplittingContProblemKL}, which was given by
\begin{subequations}
\begin{align*}
\partial_t \mathbf K(t,x) =& -\mathbf{W}^T\mathbf{A}\mathbf W\partial_x \mathbf K,\\
\dot{\mathbf S}(t) =& \mathbf{W}^T\mathbf{A}\mathbf W\mathbf S^T \langle\partial_x \mathbf X \mathbf X^{T}\rangle  ,\\
\dot{\mathbf L}(t) =& -\mathbf{A}\mathbf L\langle\partial_x \mathbf X\mathbf X^T\rangle .
\end{align*}
\end{subequations}
Now, we wish to discretize the above system. Note that one only has to solve one hyperbolic partial differential equation and two ordinary differential equations. Hence, we only need to perform a finite volume discretization for the $K$ equation. As also observed in \cite{PeMF20}, we do not need to use stabilizing numerical fluxes in the approximation of spatial derivatives in the $S$ and $L$ steps. The reason for this is that the derivatives only enter as averages (i.e.~in integrated form). Here we use the standard central second order difference stencil $\partial_x X_k(t_n,x)\big\vert_{x_j}\approx\frac{1}{2\Delta x}(X_{j+1,k}^n-X_{j-1,k}^n)$ in the last two steps. The numerical scheme then becomes
\begin{subequations}\label{eq:projSplittingFirstDLR}
\begin{align}
    \mathbf{K}^{n+1} &= \mathbf L^{(1)}\mathbf K^{n}-\mathbf L^{(2)}\mathbf K^{n}\mathbf{\tilde{A}}, \\
    \mathbf{\widetilde S}^{n+1} &= \mathbf{\widetilde S}^{n}+\mathbf X^{n+1,T}\mathbf L^{(2)}\mathbf X^{n+1}\mathbf{\widetilde S}^{n}\mathbf{\tilde{A}}, \label{eq:projSplittingFirstDLRS}\\
    \mathbf{L}^{n+1} &= \mathbf L^{n}-\mathbf X^{n+1,T}\mathbf L^{(2)}\mathbf X^{n+1}\mathbf{\widetilde S}^{n+1}\mathbf{W}^{n,T}\mathbf{A}^T.
\end{align}
\end{subequations}
Note that such a discretization has been discussed in \cite{PeMF20,kusch2021DLRUQ}. Let us investigate $L^2$-stability for the above system. Our main result is summarized in
\begin{theorem}\label{th:firstDRL}
Assume that the CFL condition
\[ \lambda_{max}(\mathbf{\tilde{A}})\Delta t/\Delta x \leq 1/\sqrt{3} \] 
holds true. Then, the projector-splitting scheme \eqref{eq:projSplittingFirstDLR} is $L^2$-stable, i.e., 
\begin{align*}
  \Vert \mathbf{u}^{n+1} \Vert_F \leq  \Vert \mathbf{u}^{n} \Vert_F.
\end{align*}
%Assume that with $\tilde c := \lambda_{max}(\mathbf{\tilde{A}})\Delta t/\Delta x$ the CFL condition
%\begin{align*}
%\left(1+\tilde c^2\sin^2(\alpha\pi\Delta x)\right)\cdot\sqrt{\cos^2(\alpha\pi\Delta x)+\tilde c^2\sin^2(\alpha\pi\Delta x)}\leq 1
%\end{align*}
%holds for $\alpha = 1,\cdots,N_x$. Then, the projector-splitting scheme \eqref{eq:projSplittingFirstDLR} is $L^2$-stable, i.e., 
%\begin{align*}
%  \Vert \mathbf{u}^{n+1} \Vert_F \leq  \Vert \mathbf{u}^{n} \Vert_F.
%\end{align*}
\end{theorem} 
\begin{proof}
We again write \eqref{eq:projSplittingFirstDLR} in terms of $\mathbf{u}_{\RomanNumeralCaps{1}},\mathbf{u}_{\RomanNumeralCaps{2}}$ and $\mathbf{u}_{\RomanNumeralCaps{3}}$. This yields
\begin{subequations}\label{eq:projSplittingMatrixNotationFirstDLR}
\begin{align}
\mathbf u_{\RomanNumeralCaps{1}}^{n+1} =& \left(\mathbf L^{(1)}\mathbf u^{n}_{\RomanNumeralCaps{1}}-\mathbf L^{(2)}\mathbf u^{n}_{\RomanNumeralCaps{1}}\mathbf{A}^T\right)\mathbf{W}^n\mathbf{W}^{n,T},\\
\mathbf u_{\RomanNumeralCaps{2}}^{n+1} =&\mathbf{X}^{n+1}\mathbf{X}^{n+1,T}\left(\mathbf u^{n}_{\RomanNumeralCaps{2}}+\mathbf L^{(2)}\mathbf u^{n}_{\RomanNumeralCaps{2}}\mathbf{A}^T\right)\mathbf{W}^n\mathbf{W}^{n,T},\\
\mathbf u_{\RomanNumeralCaps{3}}^{n+1} =& \mathbf{X}^{n+1}\mathbf{X}^{n+1,T}\left(\mathbf u^{n}_{\RomanNumeralCaps{3}}-\mathbf L^{(2)}\mathbf u^{n}_{\RomanNumeralCaps{3}}\mathbf{A}^T\right).
\end{align}
\end{subequations}
As before, all substeps can be brought into the form $\mathbf{u}^n = \mathbf{E}_x\mathbf{\tilde u}^n\mathbf{W}^{n,T}$, where $\mathbf{\tilde u}^n \in\mathbb{R}^{N_x\times r}$. Then, the split equations \eqref{eq:projSplittingMatrixNotationFirstDLR} become
\begin{subequations}\label{eq:projSplittingFourierNew}
\begin{align}
\mathbf u_{\RomanNumeralCaps{1}}^{n+1} =&\mathbf{E}_x\left(\mathbf{\tilde{u}}^{n}_{\RomanNumeralCaps{1}}-\mathbf D^{(2)}\mathbf{\tilde{u}}^n_{\RomanNumeralCaps{1}}\mathbf{\tilde{A}}\right)\mathbf{W}^{n,T},\label{eq:projSplittingFourier12}\\
\mathbf u_{\RomanNumeralCaps{2}}^{n+1} =& \mathbf{X}^{n+1}\mathbf{X}^{n+1,T}\mathbf{E}_x\left(\mathbf{\tilde{u}}^{n}_{\RomanNumeralCaps{2}}+\mathbf D^{(2)}\mathbf{\tilde{u}}^{n}_{\RomanNumeralCaps{2}}\mathbf{\tilde{A}}\right)\mathbf{W}^{n,T},\label{eq:projSplittingFourier22}\\
\mathbf u_{\RomanNumeralCaps{3}}^{n+1} =& \mathbf{X}^{n+1}\mathbf{X}^{n+1,T}\mathbf{E}_x\left(\mathbf{\tilde{u}}^{n}_{\RomanNumeralCaps{3}}-\mathbf D^{(2)}\mathbf{\tilde{u}}^{n}_{\RomanNumeralCaps{3}}\mathbf{\tilde{A}}\right)\mathbf{W}^{n,T}.\label{eq:projSplittingFourier32}
\end{align}
\end{subequations}
Now, we derive an upper bound for the norm of every substep in \eqref{eq:projSplittingFourierNew}. Let us start with the first step, which gives
\begin{align}\label{eq:estimateStepI}
\Vert \mathbf{u}_{\RomanNumeralCaps{1}}^{n+1} \Vert_F \leq \left\Vert\mathbf{E}_x\right\Vert\left\Vert\mathbf D^{(1)}\mathbf{\tilde{u}}^{n}_{\RomanNumeralCaps{1}}-\mathbf D^{(2)}\mathbf{\tilde{u}}^{n}_{\RomanNumeralCaps{1}}\mathbf{\tilde{A}}\right\Vert_F\cdot\Vert\mathbf{W}^{n,T}\Vert.
\end{align}
Following the derivation of \eqref{eq:FrobeniusNormEstimate}, we have that
\begin{align*}
\left\Vert\mathbf D^{(1)}\mathbf{\tilde{u}}^{n}_{\RomanNumeralCaps{1}}-\mathbf D^{(2)}\mathbf{\tilde{u}}^{n}_{\RomanNumeralCaps{1}}\mathbf{\tilde{A}}\right\Vert_F \leq \max_{\alpha}\left|D^{(1)}_{\alpha\alpha}-D^{(2)}_{\alpha\alpha}\lambda_{max}(\mathbf{\tilde{A}})\right|\cdot\Vert \mathbf{\tilde{u}}_{\RomanNumeralCaps{1}}^{n} \Vert_F.
\end{align*}
Since $\Vert\mathbf{W}^{n,T}\Vert=\left\Vert\mathbf{E}_x\right\Vert=1$ and by Parseval's identity $\Vert \mathbf{u}_{\RomanNumeralCaps{1}}^{n} \Vert_F=\Vert  \mathbf{E}_x\mathbf{\tilde u}^n\mathbf{W}^{n,T}\Vert_F=\Vert\mathbf{\tilde u}^n\Vert_F$ the estimate \eqref{eq:estimateStepI} becomes
\begin{align*}
\Vert \mathbf{u}_{\RomanNumeralCaps{1}}^{n+1} \Vert_F \leq\max_{\alpha}\left|D^{(1)}_{\alpha\alpha}-D^{(2)}_{\alpha\alpha}\lambda_{max}(\mathbf{\tilde{A}})\right|\cdot\Vert \mathbf{u}_{\RomanNumeralCaps{1}}^{n} \Vert_F.
\end{align*}

Proceeding in the same manner for \eqref{eq:projSplittingFourier22} and \eqref{eq:projSplittingFourier32}, we obtain the upper bounds
\begin{align*}
\Vert \mathbf{u}_{\RomanNumeralCaps{1}}^{n+1} \Vert_F \leq&\max_{\alpha}\left|D^{(1)}_{\alpha\alpha}-D^{(2)}_{\alpha\alpha}\lambda_{max}(\mathbf{\tilde{A}})\right|\cdot\Vert \mathbf{u}_{\RomanNumeralCaps{1}}^{n} \Vert_F,\\
\Vert \mathbf{u}_{\RomanNumeralCaps{2}}^{n+1} \Vert_F \leq&\max_{\alpha}\left|1+D^{(2)}_{\alpha\alpha}\lambda_{max}(\mathbf{\tilde{A}})\right|\cdot\Vert \mathbf{u}_{\RomanNumeralCaps{1}}^{n+1} \Vert_F,\\
\Vert \mathbf{u}_{\RomanNumeralCaps{3}}^{n+1} \Vert_F \leq&\max_{\alpha}\left|1-D^{(2)}_{\alpha\alpha}\lambda_{max}(\mathbf{\tilde{A}})\right|\cdot\Vert \mathbf{u}_{\RomanNumeralCaps{2}}^{n+1} \Vert_F.
\end{align*}
The amplification of $\mathbf{u}^{n+1} = \mathbf u_{\RomanNumeralCaps{3}}^{n+1}$ then satisfies 
\begin{align*}
\Vert \mathbf{u}^{n+1} \Vert_F \leq\max_{\alpha}\left|D^{(1)}_{\alpha\alpha}-D^{(2)}_{\alpha\alpha}\lambda_{max}(\mathbf{\tilde{A}})\right|\cdot\max_{\alpha}\left|1-D^{(2)}_{\alpha\alpha}\lambda_{max}(\mathbf{\tilde{A}})\right|\cdot\max_{\alpha}\left|1+D^{(2)}_{\alpha\alpha}\lambda_{max}(\mathbf{\tilde{A}})\right|\cdot\Vert \mathbf{u}^{n} \Vert_F.
\end{align*}
With $\tilde c = \lambda_{max}(\mathbf{\tilde{A}})\Delta t/\Delta x$ we have
\begin{align*}
\left|D^{(1)}_{\alpha\alpha}-D^{(2)}_{\alpha\alpha}\lambda_{max}(\mathbf{\tilde{A}})\right| = \sqrt{\cos^2(\alpha\pi\Delta x)+\tilde c^2\sin^2(\alpha\pi\Delta x)}
\end{align*}
and
\begin{align*}
\left|1\pm D^{(2)}_{\alpha\alpha}\lambda_{max}(\mathbf{\tilde{A}})\right| = \sqrt{1+\tilde c^2\sin^2(\alpha\pi\Delta x)}.
\end{align*}
Thus, we have
    \[  \Vert \mathbf{u}^{n+1} \Vert_F \leq G \Vert \mathbf{u}^{n} \Vert_F \]
with
    \[    \qquad G = \left(1+\tilde c^2\sin^2(\alpha\pi\Delta x)\right) \sqrt{\cos^2(\alpha\pi\Delta x)+\tilde c^2\sin^2(\alpha\pi\Delta x)}. \]
Since $G^2$ is a trigonometric polynomial of degree $3$ we can easily determine the stated bound.

\end{proof}
\begin{remark}
It becomes clear that since the $S$-step goes backward in time, the dampening effects of the spatial discretization (or the diffusion effects arsing from artificial viscosity) will lead to an amplification. This is the case when discretizing first and applying DLRA second. Removing this effect in the $S$-step gives us a stable scheme, as has been shown. We note, however, that the CFL condition is slightly more restrictive compared to what we would expect if no low-rank approximation is performed. To remedy this deficiency is the purpose of the remainder of this section.
\end{remark}

In the following, we derive a discretization of the continuous DLRA formulation \eqref{eq:projSplittingContProblem}. The discretization has the same CFL condition as the original problem and the unconventional integrator (to be discussed in the next section). Let us notice that the framework of performing the dynamical low-rank approximation first and discretizing second allows us to add stabilization directly into the $S$-step. For the $K$ and $L$ equations, we use Lax-Friedrichs numerical fluxes. In this case, we recover the $K$ and $L$ equations from the discretize first ansatz \eqref{eq:projSplittingMatrixNotationKSL}. The stabilized equations read
\begin{subequations}\label{eq:projSplittingFirstDLRStabilized}
\begin{align}
    \mathbf{K}^{n+1} &= \mathbf L^{(1)}\mathbf K^{n}-\mathbf L^{(2)}\mathbf K^{n}\mathbf{\tilde{A}}, \\
    \mathbf{\widetilde S}^{n+1} &= \mathbf{X}^{n+1,T}\mathbf L^{(1)}\mathbf{X}^{n+1}\mathbf{\widetilde S}^{n}+\mathbf X^{n+1,T}\mathbf L^{(2)}\mathbf X^{n+1}\mathbf{\widetilde S}^{n}\mathbf{\tilde{A}}, \\
    \mathbf{L}^{n+1} &= \mathbf{X}^{n+1,T}\mathbf L^{(1)}\mathbf{X}^{n+1}\mathbf{L}^{n}-\mathbf X^{n+1,T}\mathbf L^{(2)}\mathbf X^{n+1}\mathbf{L}^{n}\mathbf{A}^T.
\end{align}
\end{subequations}
The main difference to the previously discussed $S$-step discretization \eqref{eq:projSplittingFirstDLRS} is using the term $\mathbf{X}^{n+1,T}\mathbf L^{(1)}\mathbf{X}^{n+1}\mathbf{\widetilde S}^{n}$ instead of $\mathbf{\widetilde S}^{n}$. This term stems from adding a stabilization terms in the finite volume discretization that is used for the $S$-step. Opposed to the $S$-step of the discrete DLRA approach \eqref{eq:projSplittingMatrixNotationS}, we use a negative sign in front of the stabilization term. With $\mathbf u^{n}_{\RomanNumeralCaps{2}} = \mathbf{X}^{n+1,T}\mathbf{\widetilde S}^{n}\mathbf{W}^n$, we hence choose
\begin{align*}
    \mathbf{\widetilde S}^{n+1} =& \mathbf{\widetilde S}^{n}- \mathbf{X}^{n+1,T}\left((\mathbf I - \mathbf L^{(1)})\mathbf u^{n}_{\RomanNumeralCaps{2}}-\mathbf L^{(2)}\mathbf u^{n}_{\RomanNumeralCaps{2}}\mathbf{A}^T\right)\mathbf{W}^n \\
    =& \mathbf{X}^{n+1,T}\left(\mathbf L^{(1)}\mathbf u^{n}_{\RomanNumeralCaps{2}}+\mathbf L^{(2)}\mathbf u^{n}_{\RomanNumeralCaps{2}}\mathbf{A}^T\right)\mathbf{W}^n \\
    =& \mathbf{X}^{n+1,T}\mathbf L^{(1)}\mathbf{X}^{n+1}\mathbf{\widetilde S}^{n}+\mathbf X^{n+1,T}\mathbf L^{(2)}\mathbf X^{n+1}\mathbf{\widetilde S}^{n}\mathbf{\tilde{A}}.
\end{align*}
Since the stabilization term does not affect consistency, changing its sign will preserve consistency of our scheme.
For the presented scheme \eqref{eq:projSplittingFirstDLRStabilized}, we have the following stability result:
\begin{theorem}\label{th:unconventional}
If the CFL condition 
    \[ \lambda_{max}(\mathbf{\tilde{A}})\Delta t/\Delta x \leq 1\]
holds, the projector-splitting scheme \eqref{eq:projSplittingFirstDLRStabilized} is $L^2$-stable, i.e., 
\begin{align*}
  \Vert \mathbf{u}^{n+1} \Vert_F \leq  \Vert \mathbf{u}^{n} \Vert_F.
\end{align*}
\end{theorem} 
\begin{proof}
The proof follows that of Theorem~\ref{th:firstDRL} with the main difference that the split equations corresponding to \eqref{eq:projSplittingFirstDLRStabilized} read
\begin{subequations}
\begin{align*}
\mathbf u_{\RomanNumeralCaps{1}}^{n+1} =& \left(\mathbf L^{(1)}\mathbf u^{n}_{\RomanNumeralCaps{1}}-\mathbf L^{(2)}\mathbf u^{n}_{\RomanNumeralCaps{1}}\mathbf{A}^T\right)\mathbf{W}^n\mathbf{W}^{n,T},\\
\mathbf u_{\RomanNumeralCaps{2}}^{n+1} =&\mathbf{X}^{n+1}\mathbf{X}^{n+1,T}\left(\mathbf L^{(1)}\mathbf u^{n}_{\RomanNumeralCaps{2}}+\mathbf L^{(2)}\mathbf u^{n}_{\RomanNumeralCaps{2}}\mathbf{A}^T\right)\mathbf{W}^n\mathbf{W}^{n,T},\\
\mathbf u_{\RomanNumeralCaps{3}}^{n+1} =& \mathbf{X}^{n+1}\mathbf{X}^{n+1,T}\left(\mathbf L^{(1)}\mathbf u^{n}_{\RomanNumeralCaps{3}}-\mathbf L^{(2)}\mathbf u^{n}_{\RomanNumeralCaps{3}}\mathbf{A}^T\right).
\end{align*}
\end{subequations}
Following the derivation from Theorem~\ref{th:firstDRL}, we have that
\begin{align*}
\Vert \mathbf{u}_{\RomanNumeralCaps{1}}^{n+1} \Vert_F \leq&\max_{\alpha}\left|D^{(1)}_{\alpha\alpha}-D^{(2)}_{\alpha\alpha}\lambda_{max}(\mathbf{\tilde{A}})\right|\cdot\Vert \mathbf{u}_{\RomanNumeralCaps{1}}^{n} \Vert_F,\\
\Vert \mathbf{u}_{\RomanNumeralCaps{2}}^{n+1} \Vert_F \leq&\max_{\alpha}\left|D^{(1)}_{\alpha\alpha}+D^{(2)}_{\alpha\alpha}\lambda_{max}(\mathbf{\tilde{A}})\right|\cdot\Vert \mathbf{u}_{\RomanNumeralCaps{1}}^{n+1} \Vert_F,\\
\Vert \mathbf{u}_{\RomanNumeralCaps{3}}^{n+1} \Vert_F \leq&\max_{\alpha}\left|D^{(1)}_{\alpha\alpha}-D^{(2)}_{\alpha\alpha}\lambda_{max}(\mathbf{\tilde{A}})\right|\cdot\Vert \mathbf{u}_{\RomanNumeralCaps{2}}^{n+1} \Vert_F.
\end{align*}
With $\tilde{c} = \lambda_{max}(\mathbf{\tilde{A}})\Delta t/\Delta x$ this yields
\begin{align}\label{eq:amplificationFactorPSstable}
\Vert \mathbf{u}^{n+1} \Vert_F \leq&\max_{\alpha}\left|D^{(1)}_{\alpha\alpha}-D^{(2)}_{\alpha\alpha}\lambda_{max}(\mathbf{\tilde{A}})\right|^3\cdot\Vert \mathbf{u}^{n} \Vert_F\nonumber\\
=&\left(\cos^2(\alpha\pi\Delta x)+\tilde c^2\sin^2(\alpha\pi\Delta x)\right)^{3/2}\cdot\Vert \mathbf{u}^{n} \Vert_F.
\end{align}
\end{proof}

\section{$L^2$-stability analysis for the unconventional integrator}\label{sec:L2StabilityDLRAdiscreteUnc}
Let us now investigate $L^2$-stability of the unconventional integrator. We apply dynamical low-rank approximation to the fully discretized matrix ODE \eqref{eq:schemeMatrixNotationSemiDiscrete}.
%\begin{enumerate}
%    \item \textbf{$K$-step}: Update $\mathbf X^{0}$ to $\mathbf X^{1}$ via
%\begin{align*}
%\dot{\mathbf K}(t) &= -\frac{1}{\Delta t}\left((\mathbf L^{(1)}-\mathbf I)\mathbf u(t)-\mathbf L^{(2)}\mathbf u(t)\mathbf{A}^T\right)\mathbf{W}^0.
%\end{align*}
%Determine $\mathbf X^1$ with a QR-decomposition $\mathbf K(t_1) = \mathbf X^1 \mathbf R$ and store $\mathbf M = \mathbf X^{1,T}\mathbf X^0$.
%\item \textbf{$L$-step}: Update $\mathbf W^0$ to $\mathbf W^1$ via
%\begin{align*}
%\dot{\mathbf L}(t) &= -\frac{1}{\Delta t}\mathbf{X}^{0,T}\left((\mathbf L^{(1)}-\mathbf I)\mathbf u(t)-\mathbf L^{(2)}\mathbf u(t)\mathbf{A}^T\right)
%\end{align*}
%Determine $\mathbf W^1$ with a QR-decomposition $\mathbf L^1 = \mathbf W^1\mathbf{\widetilde R}$ and store $\mathbf N = \mathbf W^{1,T} \mathbf W^0$.
%\item \textbf{$S$-step}: Update $\mathbf S^0$ to $\mathbf S^1$ via
%\begin{align*}
%\dot{\mathbf S}(t) &= -\frac{1}{\Delta t}\mathbf{X}^{1,T}\left((\mathbf L^{(1)}-\mathbf I)\mathbf X^1\mathbf S\mathbf W^{1,T}-\mathbf L^{(2)}\mathbf X^1\mathbf S\mathbf W^{1,T}\mathbf{A}^T\right)\mathbf{W}^1\\
%\mathbf S(t_0) &= \mathbf M\mathbf S^0 \mathbf N^T
%\end{align*}
%and set $\mathbf S^1 = \mathbf S(t_0+\Delta t)$.
%\end{enumerate}
Using an explicit Euler time-discretization, the $K$, $L$ and $S$-steps \eqref{eq:KStepSemiDiscreteUI}, \eqref{eq:LStepSemiDiscreteUI} and \eqref{eq:SStepSemiDiscreteUI} of the unconventional integrator can be written as
\begin{subequations}\label{eq:unconventionalMatrixNotation}
\begin{align}
\mathbf K^{n+1} =& \left(\mathbf L^{(1)}\mathbf u^{n}-\mathbf L^{(2)}\mathbf u^{n}\mathbf{A}^T\right)\mathbf{W}^n,\\
\mathbf L^{n+1} =& \mathbf{X}^{n,T}\left(\mathbf L^{(1)}\mathbf u^{n}-\mathbf L^{(2)}\mathbf u^{n}\mathbf{A}^T\right), \\
\mathbf S^{n+1} =&\mathbf{X}^{n+1,T}\left(\mathbf L^{(1)}\mathbf{\bar u}-\mathbf L^{(2)}\mathbf{\bar u}\mathbf{A}^T\right)\mathbf{W}^{n+1},
\end{align}
\end{subequations}
where $\mathbf u^{n} = \mathbf X^n \mathbf S^n\mathbf W^{n,T}$ and $\mathbf{\bar u} = \mathbf X^{n+1} \mathbf M \mathbf S^n \mathbf N^T\mathbf W^{n+1,T}$. The matrices $\mathbf N$ and $\mathbf M$ are given by $\mathbf M = \mathbf X^{n+1,T}\mathbf X^n$ and $\mathbf N = \mathbf W^{n+1,T}\mathbf{W}^n$ and we obtain $\mathbf X^{n+1}$ and $\mathbf W^{n+1}$ from QR-decompositions of $\mathbf{K}^{n+1}$ and $\mathbf{L}^{n+1}$.

Let us start investigating $L^2$-stability by again rewriting $\mathbf u^{n} = \mathbf E_x \mathbf E_x^H \mathbf X^n \mathbf{S}^n\mathbf W^{n,T}$. As for the projector-splitting integrator, with this ansatz, the input to the first two equations is of the form $\mathbf u^n = \mathbf E_x \mathbf{\tilde u}^n\mathbf W^{n,T}$ with $\mathbf{\tilde u}^n=\mathbf E_x^H \mathbf X^n \mathbf{S}^n$. Furthermore, $\mathbf{\bar u} = \mathbf E_x \mathbf{\tilde u}_S^n\mathbf W^{n+1,T}$ with $\mathbf{\tilde u}_S^n = \mathbf E_x^H\mathbf X^{n+1}\mathbf M \mathbf S^n \mathbf N^T$. In this case, following \eqref{eq:FourierResponseSpatialScheme}, we can write \eqref{eq:unconventionalMatrixNotation} in Fourier space as follows
\begin{align*}
\mathbf K^{n+1} =& \left(\mathbf E_x\mathbf D^{(1)}\mathbf{\tilde u}^n\mathbf{W}^{n,T}-\mathbf E_x\mathbf D^{(2)}\mathbf{\tilde u}^n\mathbf{W}^{n,T}\mathbf{A}^T\right)\mathbf{W}^n,\\
\mathbf L^{n+1} =& \mathbf{X}^{n,T}\left(\mathbf E_x\mathbf D^{(1)}\mathbf{\tilde u}^n\mathbf{W}^{n,T}-\mathbf E_x\mathbf D^{(2)}\mathbf{\tilde u}^n\mathbf{W}^{n,T}\mathbf{A}^T\right), \\
\mathbf S^{n+1} =&\mathbf{X}^{n+1,T}\left(\mathbf E_x\mathbf D^{(1)}\mathbf{\tilde u}_S^n\mathbf W^{n+1,T}-\mathbf E_x\mathbf D^{(2)}\mathbf{\tilde u}_S^n\mathbf W^{n+1,T}\mathbf{A}^T\right)\mathbf{W}^{n+1},
\end{align*}
After a few simplifications and making use of $\mathbf{\tilde{A}}^n :=\mathbf{W}^{n,T}\mathbf{A}^T\mathbf{W}^{n}$, we have
\begin{subequations}\label{eq:unconventionalFourierResponse}
\begin{align}
\mathbf K^{n+1} =& \mathbf E_x\mathbf D^{(1)}\mathbf{\tilde u}^n-\mathbf E_x\mathbf D^{(2)}\mathbf{\tilde u}^n\mathbf{\tilde{A}}^n,\\
\mathbf L^{n+1} =& \mathbf{X}^{n,T}\left(\mathbf E_x\mathbf D^{(1)}\mathbf{\tilde u}^n-\mathbf E_x\mathbf D^{(2)}\mathbf{\tilde u}^n\mathbf{\tilde{A}}^n\right)\mathbf{W}^{n,T}, \\
\mathbf S^{n+1} =&\mathbf{X}^{n+1,T}\left(\mathbf E_x\mathbf D^{(1)}\mathbf{\tilde u}_S^n-\mathbf E_x\mathbf D^{(2)}\mathbf{\tilde u}_S^n\mathbf{\tilde{A}}^{n+1}\right)\label{eq:unconventionalFourierResponseS}.
\end{align}
\end{subequations}
This representation allows an easy verification of the following Theorem:
\begin{theorem}
If the CFL condition
\begin{align*}
    \lambda_{max}(\mathbf{\tilde{A}})\frac{\Delta t}{\Delta x}\leq1
\end{align*}
holds, the unconventional integrator \eqref{eq:unconventionalMatrixNotation} is $L^2$-stable, i.e.,
    \[ \Vert \mathbf{u}^{n+1} \Vert_F  \leq \Vert \mathbf{u}^{n} \Vert_F. \]
\end{theorem}
\begin{proof}
The Frobenius norm of the time updated solution is given by
\begin{align*}
    \Vert \mathbf{u}^{n+1} \Vert_F = \left\Vert \mathbf{X}^{n+1}\mathbf S^{n+1}\mathbf W^{n+1,T} \right\Vert_F = \left\Vert \mathbf S^{n+1} \right\Vert_F.
\end{align*}
Taking the norm of \eqref{eq:unconventionalFourierResponseS} gives
\begin{align*}
    \Vert \mathbf S^{n+1} \Vert\leq \Vert \mathbf{X}^{n+1,T}\Vert\cdot\Vert\mathbf E_x\Vert\cdot\max_{\alpha}\left|D^{(1)}_{\alpha\alpha}-D^{(2)}_{\alpha\alpha}\lambda_{max}(\mathbf{\tilde{A}})\right|\cdot\Vert \mathbf{\tilde u}_S^n \Vert_F.
\end{align*}
Furthermore, we have
\begin{align*}
    \Vert \mathbf{\tilde u}_S^n \Vert_F = \Vert \mathbf E_x^H\mathbf{X}^{n+1}\mathbf M \mathbf S^n \mathbf N^T\mathbf W^{n+1,T} \Vert_F \leq& \Vert \mathbf M\Vert\cdot\Vert \mathbf S^n \Vert_F\cdot\Vert\mathbf N^T\Vert \\
    =& \Vert \mathbf X^{n+1,T}\mathbf X^{n}\Vert\cdot\Vert \mathbf W^{n,T}\mathbf W^{n+1}\Vert\cdot\Vert \mathbf S^n \Vert_F \\
    \leq& \Vert \mathbf S^n \Vert_F,
\end{align*}
where we used that $\Vert \mathbf X^{n+1,T}\mathbf X^{n}\Vert\leq \Vert \mathbf X^{n+1,T}\Vert\cdot\Vert \mathbf X^{n}\Vert= 1$. Hence, we obtain 
\begin{align*}
    \Vert \mathbf{u}^{n+1} \Vert_F \leq& \max_{\alpha}\left|D^{(1)}_{\alpha\alpha}-D^{(2)}_{\alpha\alpha}\lambda_{max}(\mathbf{\tilde{A}})\right|\cdot\Vert \mathbf S^n \Vert_F \\
    =&\max_{\alpha}\left|D^{(1)}_{\alpha\alpha}-D^{(2)}_{\alpha\alpha}\lambda_{max}(\mathbf{\tilde{A}})\right|\cdot\Vert \mathbf{u}^{n} \Vert_F.
\end{align*}
Again, we have that
\begin{align}\label{eq:factorAbove}
\left|D^{(1)}_{\alpha\alpha}-D^{(2)}_{\alpha\alpha}\lambda_{max}(\mathbf{\tilde{A}})\right| = \sqrt{\cos^2(\alpha\pi\Delta x)+\tilde c^2\sin^2(\alpha\pi\Delta x)}
\end{align}
with $\tilde c := \lambda_{max}(\mathbf{\tilde{A}})\Delta t/\Delta x$. Hence,
\begin{align}\label{eq:amplificationFactorUnconventional}
\Vert \mathbf{u}^{n+1} \Vert_F \leq&\left(\cos^2(\alpha\pi\Delta x)+\tilde c^2\sin^2(\alpha\pi\Delta x)\right)^{1/2}\cdot\Vert \mathbf{u}^{n} \Vert_F.
\end{align}
To ensure that the factor \eqref{eq:factorAbove} in the above expression is bounded by one, we need to choose $\tilde c\leq 1$, which proves the theorem.
\end{proof}
\begin{remark}\label{rem:dampeningPSvsUnconv}
It is worth noting that the amplification factor in \eqref{eq:amplificationFactorPSstable} will be smaller than the dampening for the unconventional integrator which is given in \eqref{eq:amplificationFactorUnconventional}. This is an advantage of the unconventional integrator, as it adds less artificial diffusion to guarantee a stable scheme. We will discuss this in more detail in section \ref{sec:results}.
\end{remark}

\section{Scattering}\label{sec:scattering}
In this section, we include scattering terms that arise in kinetic transport problems. For this, we perform a dynamical low-rank approximation for the streaming and scattering equations \eqref{eq:streamingScatteringEqns}. I.e., we obtain one set of the $K$, $S$ and $L$ steps for the streaming part and one set for the scattering part. Our approach shares similarities with the method proposed in \cite{ostermann2019convergence}, where stiff and non-stiff parts of the original equation are seperated through a splitting step. It is straightforward to show that the unconventional integrator again provides a stable scheme. Therefore, we directly investigate the matrix projector-splitting integrator. Since we already discussed stability for the streaming equations, we first write down the split equations \eqref{eq:projectorSplitEq} for the scattering equations. To distinguish from the streaming solution, let us use Arabic instead of Roman numbers to denote substeps of the projector-splitting integrator. Using a forward Euler time discretization we have 
\begin{subequations}\label{eq:projSplittingMatrixNotationScattering}
\begin{align}
\mathbf u_{1}^{n+1} =& \mathbf u^{n+1/2}_{1}\left(\mathbf I + \Delta t \mathbf G\right)\mathbf{W}^{n+1/2}\mathbf{W}^{n+1/2,T},\\
\mathbf u_{2}^{n+1} =&\mathbf{X}^{n+1}\mathbf{X}^{n+1,T}\mathbf u^{n+1}_{1}\left(\mathbf I - \Delta t \mathbf G\right)\mathbf{W}^{n+1/2}\mathbf{W}^{n+1/2,T},\\
\mathbf u_{3}^{n+1} =& \mathbf{X}^{n+1}\mathbf{X}^{n+1,T}\mathbf u^{n+1}_{2}\left(\mathbf I + \Delta t \mathbf G\right).
\end{align}
\end{subequations}
Written as a single expression, this gives
\begin{align*}
    \mathbf u_{3}^{n+1} = &\mathbf{X}^{n+1}\mathbf{X}^{n+1,T}\mathbf u^{n+1/2}_{1}\left(\mathbf I + \Delta t \mathbf G\right)\mathbf{W}^{n+1/2}\\
    &\cdot\mathbf{W}^{n+1/2,T}\left(\mathbf I - \Delta t \mathbf G\right)\mathbf{W}^{n+1/2}\mathbf{W}^{n+1/2,T}\left(\mathbf I + \Delta t \mathbf G\right).
\end{align*}
Hence, the amplification is again 
\begin{align*}
\Vert \mathbf u_{3}^{n+1} \Vert_F \leq \max_{\ell}\left|1+\Delta t G_{\ell\ell}\right|^2 \max_{\ell}\left|1-\Delta t G_{\ell\ell}\right|\Vert \mathbf u_{\RomanNumeralCaps{3}}^{n+1/2} \Vert_F.
\end{align*}
This implies that the scheme is stable as long as $\max_{\ell}\vert G_{\ell \ell}\vert \Delta t  \leq 1.62$. We note that this stability constraints is more severe than the explicit Euler scheme applied to the original equation (i.e.~without performing a low-rank approximation).

In the following, we propose a discretization of the scattering part that recovers the classic CFL condition. For this, we go one step back and start from the time-continuous $K$, $S$ and $L$-equations of the scattering step \eqref{eq:projSplittingMatrixNotationScattering} which read
\begin{subequations}
\begin{align}
\mathbf{\dot{L}} =& \mathbf{L}\mathbf G, \quad \mathbf{L}(t_0) = \mathbf{S}^{n+1/2}\mathbf{W}^{n+1/2,T}\\
\mathbf{\dot S} =&-\mathbf{S}\mathbf{W}^{n+1,T}\mathbf G\mathbf{W}^{n+1},\quad \mathbf{S}(t_0)\mathbf{W}^{n+1,T} = \mathbf{L}(t_0+\Delta t)\label{eq:scatteringCont2}\\
\mathbf{\dot{K}} =& \mathbf K\mathbf{W}^{n+1,T}\mathbf G\mathbf{W}^{n+1}, \quad \mathbf{K}(t_0) = \mathbf{X}^{n+1/2}\mathbf{S}(t_0+\Delta t).
\end{align}
\end{subequations}
Note that we now do the $L$-step first and the $K$-step last. Let us use $\mathbf{\tilde G} := \mathbf{W}^{n+1,T}\mathbf G\mathbf{W}^{n+1}$ and multiply \eqref{eq:scatteringCont2} with $\mathbf{X}^{n+1/2}$. Then, since $\mathbf X$ remains constant in the $S$-step, the $S$ and $K$-steps become
\begin{align*}
\mathbf{\dot{\widehat{K}}} =&-\mathbf{\widehat K}\mathbf{\widetilde G},\quad \mathbf{\widehat K}(t_0) = \mathbf{X}^{n+1/2}\mathbf{S}(t_0)\\
\mathbf{\dot{K}} =& \mathbf K\mathbf{\widetilde G}, \quad \mathbf{K}(t_0) = \mathbf{\widehat K}(t_0+\Delta t).
\end{align*}
This system of ODEs can be solved analytically through matrix exponentials
\begin{align*}
\mathbf{K}(t+\Delta t) = e^{\mathbf{\widetilde{G}}\Delta t} \mathbf{K}(t_0) = e^{\mathbf{\widetilde{G}}\Delta t}e^{-\mathbf{\widetilde{G}}\Delta t}\mathbf{\widehat K}(t_0) = \mathbf{\widehat K}(t_0).
\end{align*}
Hence, on a continuous level, the $K$ and $S$-steps cancel each other out and the dynamics is solely given by the $L$-step. Therefore, it is sufficient to only perform the $L$-step for scattering, i.e., scattering only effects the $\mathbf{S}$ and $\mathbf{W}$ factors of the solution. Using an explicit Euler time-discretization, we have
\begin{align*}
\mathbf L^{n+1} =& \mathbf{L}^{n+1/2}\left(\mathbf I + \Delta t \mathbf G\right).
\end{align*}
This gives 
\begin{align*}
\Vert \mathbf u_{3}^{n+1} \Vert_F \leq \max_{\ell}\left|1+\Delta t G_{\ell\ell}\right|\cdot\Vert \mathbf u_{\RomanNumeralCaps{3}}^{n+1/2} \Vert_F,
\end{align*}
and we thus recover the classic stability constraint given by $\max_{\ell}\vert G_{\ell \ell}\vert \Delta t  \leq 2$. For sake of completeness, let us state the full algorithm:
\begin{enumerate}
    \item \textbf{$K$-step streaming}: Update $\mathbf X^{n+1}\leftarrow\mathbf X^{n}$ and $\mathbf{\widetilde{S}}^n\leftarrow \mathbf{S}^n$ via
\begin{align*}
\mathbf K^{n+1} &= \mathbf L^{(1)}\mathbf K^{n} - \mathbf L^{(2)}\mathbf u(t)\mathbf{A}^T\mathbf{W}^n
\end{align*}
Determine $\mathbf X^{n+1}$ and $\mathbf{\widetilde{S}}^n$ with a QR-decomposition $\mathbf K^{n+1} = \mathbf X^{n+1} \mathbf{\widetilde{S}}^n$.
\item \textbf{$S$-step streaming}: Update $\mathbf{\widetilde{S}}^{n+1}\leftarrow\mathbf{\widetilde{S}}^{n}$ via
\begin{align*}
\mathbf{\widetilde S}^{n+1} &= \mathbf{X}^{n+1,T}\mathbf L^{(1)}\mathbf{X}^{n+1}\mathbf{\widetilde S}^{n}+\mathbf X^{n+1,T}\mathbf L^{(2)}\mathbf X^{n+1}\mathbf{\widetilde S}^{n}\mathbf W^{n,T}\mathbf{A}\mathbf W^n
\end{align*}
\item
\begin{itemize}
\item \textbf{$L$-step streaming}: Update $\mathbf L^{n+1/2}\leftarrow\mathbf L^{n}$ via
\begin{align*}
\mathbf{L}^{n+1/2} &= \mathbf{X}^{n+1,T}\mathbf L^{(1)}\mathbf{X}^{n+1}\mathbf{L}^{n}-\mathbf X^{n+1,T}\mathbf L^{(2)}\mathbf X^{n+1}\mathbf{L}^{n}\mathbf{A}^T.
\end{align*}
\item \textbf{$L$-step scattering}: Update to $\mathbf W^{n+1}$ and $\mathbf S^{n+1}$ from $\mathbf{L}^{n+1/2}$ via
\begin{align*}
\mathbf L^{n+1} =& \mathbf{L}^{n+1/2}\left(\mathbf I + \Delta t \mathbf G\right).
\end{align*}
\end{itemize}
Determine $\mathbf W^{n+1}$ and $\mathbf S^{n+1}$ with a QR-decomposition $\mathbf L^{n+1} = \mathbf W^{n+1}\mathbf S^{n+1}$.
\end{enumerate}

Some remarks are in order
\begin{remark}
    The strategy of splitting the original equation before applying dynamical low-rank can be applied in various situations to cancel steps in the projector-splitting integrator and thereby reduce computational costs. As an example, assume that we have an equation
\begin{align*}
    \partial_t u(t,x,y) + \mathcal{L}_1(t,x) u(t,x,y) + \mathcal{L}_2(t,y) u(t,x,y) =0,
\end{align*}
where $\mathcal{L}_1(t,x)$ is a (differential) operator which does not depend on $x$ and $\mathcal{L}_2(t,y)$ is a (differential) operator which does not depend on $y$.
We can split this equation according to
\begin{align*}
    &\partial_t u_1(t,x,y) + \mathcal{L}_1(t,x) u_1(t,x,y)=0\qquad u_1(t_0,x,y) = u(t,x,y)\\
    &\partial_t u_2(t,x,y) + \mathcal{L}_2(t,x) u_2(t,x,y) =0\qquad u_2(t_0,x,y) = u_1(t_1,x,y).
\end{align*}
Applying the projector-splitting integrator to each equation individually will then again only give an update in the $K$-step for the first equation and in the $L$-step for the second equation.
\end{remark}
\begin{remark} The proposed strategy allows for a straightforward implementation of implicit time discretization schemes for the scattering part. Since scattering can be ill-conditioned, this is an often taken approach in radiation transport. This idea has been pointed out in \cite{ostermann2019convergence} for a different setting. Here, the authors split stiff parts from the original differential equation and treat both resulting equations with adequate numerical methods.
\end{remark}

\section{Numerical results}\label{sec:results}
To allow reproducability, the code to compute all numerical results of this work is openly available \cite{code}.
\subsection{Radiation transport}
In the following, we present numerical results for the radiation transport equation, which describes the movement of radiation particles on a mesoscopic level. Particles are moving through a background medium with which they undergo collisions. In a one-dimensional setting, the particle density, also called the angular flux, is denoted by $\psi(t,x,\mu)$. Here, $t\in\mathbb{R}_+$ denotes time, $x\in [x_L,x_R]$ is the spatial variable and $\mu\in[-1,1]$ is the travelling direction of particles, projected onto a one-dimensional domain. When scattering is isotropic, the dynamics of the scalar flux can be determined from the integro-differential equation
\begin{align}\label{eq:radTransport}
    \partial_t \psi(t,x,&\mu) + \mu\partial_x \psi(t,x,\mu) + \sigma_t(x) \psi(t,x,\mu) = \frac{\sigma_s(x)}{2}\phi(t,x),\\
    &\psi(t=0,x,\mu) = \psi_{\text{IC}}(x,\mu)\\ 
    &\psi(t,x_L,\mu) = \psi_L(t,\mu) \enskip \text{ and } \enskip \psi(t,x_R,\mu) = \psi_R(t,\mu).
\end{align}
The scalar flux $\phi$ is given by $\phi(t,x) = \int_{-1}^1 \psi(t,x,\mu) \,d\mu$. Commonly, the directional dependence is represented by a modal discretization. When $P_{\ell}:[-1,1]\rightarrow \mathbb{R}$ are the normalized Legendre polynomials, the modal representation takes the form
\begin{align*}
    \psi(t,x,\mu)\approx \psi_N(t,x,\mu) := \sum_{\ell = 0}^N u_{\ell}(t,x) P_{\ell}(\mu).
\end{align*}
A Galerkin projection of the original system \eqref{eq:radTransport} yields the P$_N$ equations \eqref{eq:PN}. In this work, we study the plane source Ganapol's benchmark test~\cite{ganapol2008analytical}, which is equipped with an analytic solution. Its initial condition is an isotropic dirac distribution in the center of the spatial domain, which in numerical computations is commonly modelled as $\psi(t=0,x,\mu) = \max\{10^{-4},\sfrac{1}{\sqrt{2\pi\delta} }\exp( \sfrac{-x^2}{2\delta})\}$ using a small variance $\delta = 0.03^2$. Numerical investigations for the plane-source test-case have been conducted with dynamical low-rank approximation in \cite{PeMF20,PeM20,ceruti2021rank}. The plane-source test-case is challenging, since solutions to it are prone to numerical artifacts such as ray-effects or oscillations. Classical numerical methods for this type of problem add artificial viscosity to mitigate these spurious artifacts, see e.g. \cite{frank2020ray,mcclarren2010robust}.

As previously discussed, the dynamical low-rank approximation can either be derived for the spatially discretized P$_N$ system or for the continuous problem \eqref{eq:PN}. In the latter case, a discretization must be performed on the derived $K$, $S$ and $L$ equations. Our analysis shows stability of the unconventional integrator for both approaches, whereas the matrix-projector splitting integrator is unstable when being applied to the discretized problem. Scattering is stabilized through the splitting approach presented in Section~\ref{sec:scattering}, which for the matrix projector-splitting integrator allows for an efficient numerical treatment.

We start by studying the plane-source testcase for different integrators with ranks $10$ and $15$ as well as a CFL number of $CFL = \lambda_{max}(\mathbf A)\Delta t/\Delta x=1$. The remaining parameter values are
\begin{center}
    \begin{tabular}{ | l | p{7.0cm} |}
    \hline
    $[x_L,x_R]=[-1.5,1.5]$ & range of spatial domain \\
    $T = 1$ & end time \\
    $N_x=800$ & number of spatial cells \\
    $N+1 = 100$ & expansion coefficients in angle \\
    $\sigma_s = \sigma_t = 1$ & isotropic scattering and total cross section \\
    \hline
    \end{tabular}
\end{center}
Numerical results for these parameters are depicted in Figure~\ref{fig:Figure1} for different integrators. 
\begin{figure}[h!]
\centering
	\begin{subfigure}{0.49\linewidth}
		\centering
		\includegraphics[scale=0.21]{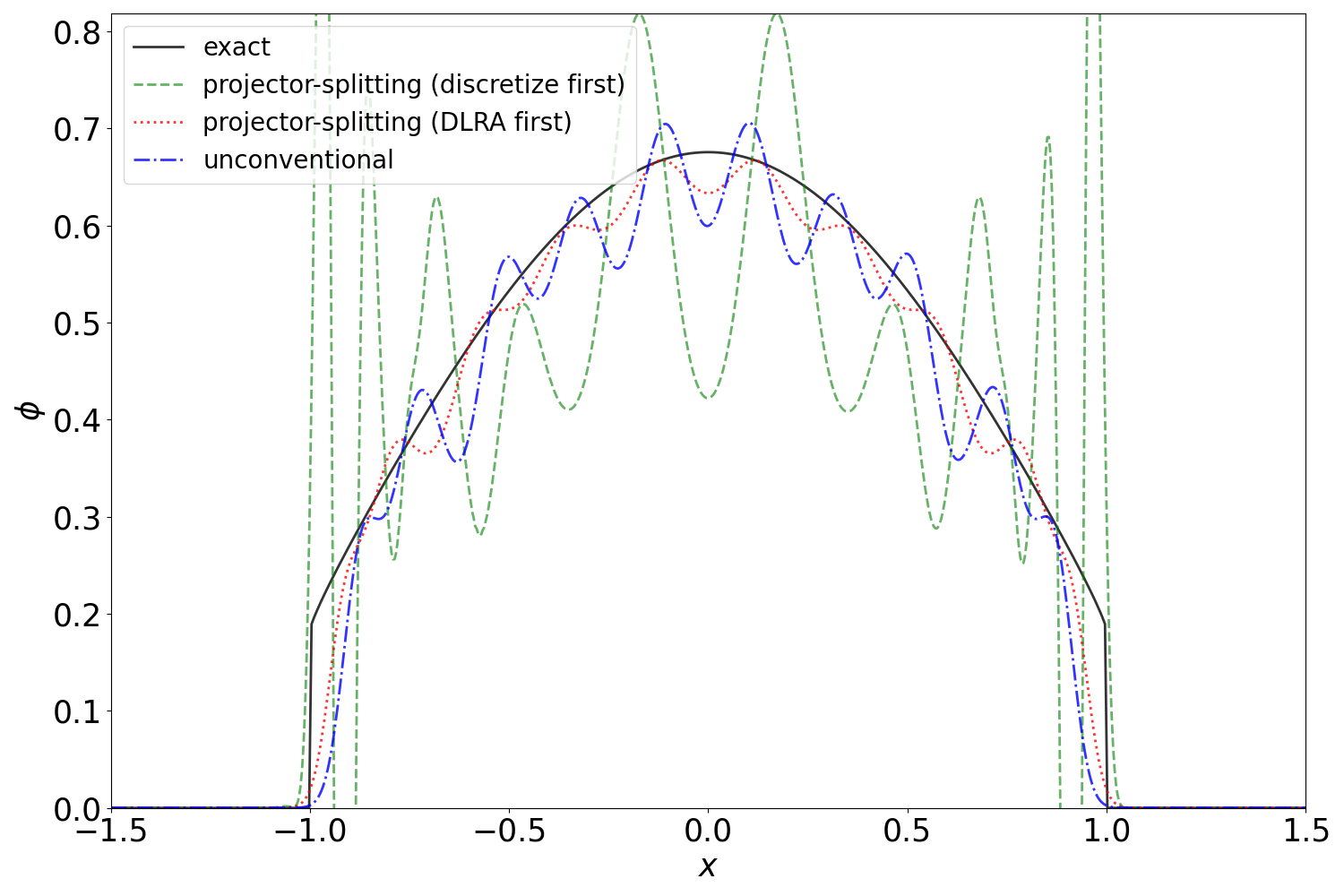}
		\caption{$r=10$}
		\label{fig:studyCFLPNRank10}
	\end{subfigure}
	\begin{subfigure}{0.49\linewidth}
		\centering
		\includegraphics[scale=0.21]{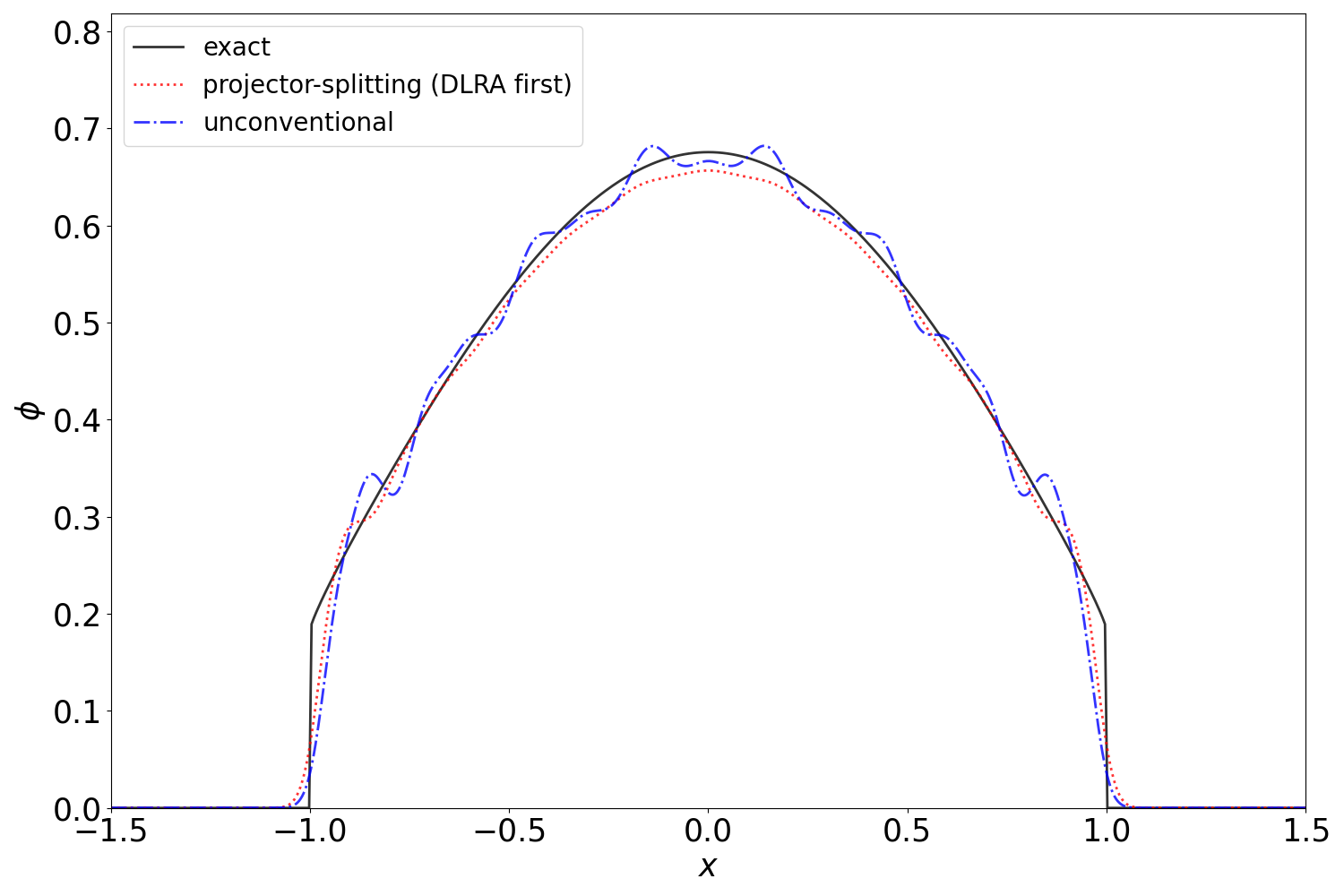}
		\caption{$r=15$}
		\label{fig:studyCFLPNRank15}
	\end{subfigure}
    \caption{Scalar flux for the plane source problem computed with the unconventional integrator, the projector-splitting integrator and its stabilized discretization for ranks $r=10$ and $r=15$. The projector-splitting integrator for the fully discretized problem yields infinite values for rank $r=15$ and is therefore not shown.}
	\label{fig:Figure1}
\end{figure}
As expected, the unconventional integrator when being applied to the fully discretized problem remains stable for this high CFL number. This is not the case for the projector-splitting integrator. In agreement with the results of Theorem~\ref{th:L2instability}, applying the projector-splitting integrator to the matrix ODE which results from discretizing the original problem does not yield an $L^2$-stable scheme. As a result, the DLRA solution when using rank $r=10$ heavily oscillates. For rank $15$, the solution blows up and the method breaks down. In our numerical experiments, we observed cases in which combinations of the matrix ODE sizes $N_x$ and $N$ lead to stable results, even for $CFL=1$. A stable discretization of the $K$, $S$ and $L$ steps of the projector-splitting integrator for the continuous problem is given by \eqref{eq:projSplittingFirstDLRStabilized}. The derived stability of this discretization can be observed in our numerical experiments. Note that this discretization appears to yield the best results of the three discussed integrators and discretizations, especially for rank $r=15$, which nicely matches the analytic solution. This results from the increased dampening of artificial viscosity for the stabilized projector-splitting integrator (cf. Remark~\ref{rem:dampeningPSvsUnconv}). Note that this increased dampening, though being beneficial for the plane-source test case, might not be desired for general problems. The analytically derived $L^2$-stability is further visualized in Figure~\ref{fig:Figure2} which depicts the Frobenius norm of the angular flux $\psi$. In agreement with the derived behaviour, the Frobenius norm is dissipated in time for the unconventional integrator and the stable discretization of the projector-splitting integrator. As expected, the dissipation of the stabilized projector-splitting integrator is stronger than for the unconventional integrator. The projector-splitting integrator when being applied on the matrix ODE of the discretized problem amplifies the norm. For rank $15$, the Frobenius norm reaches an infinite value after a few iterations.
\begin{figure}[h!]
\centering
	\begin{subfigure}{0.49\linewidth}
		\centering
		\includegraphics[scale=0.3]{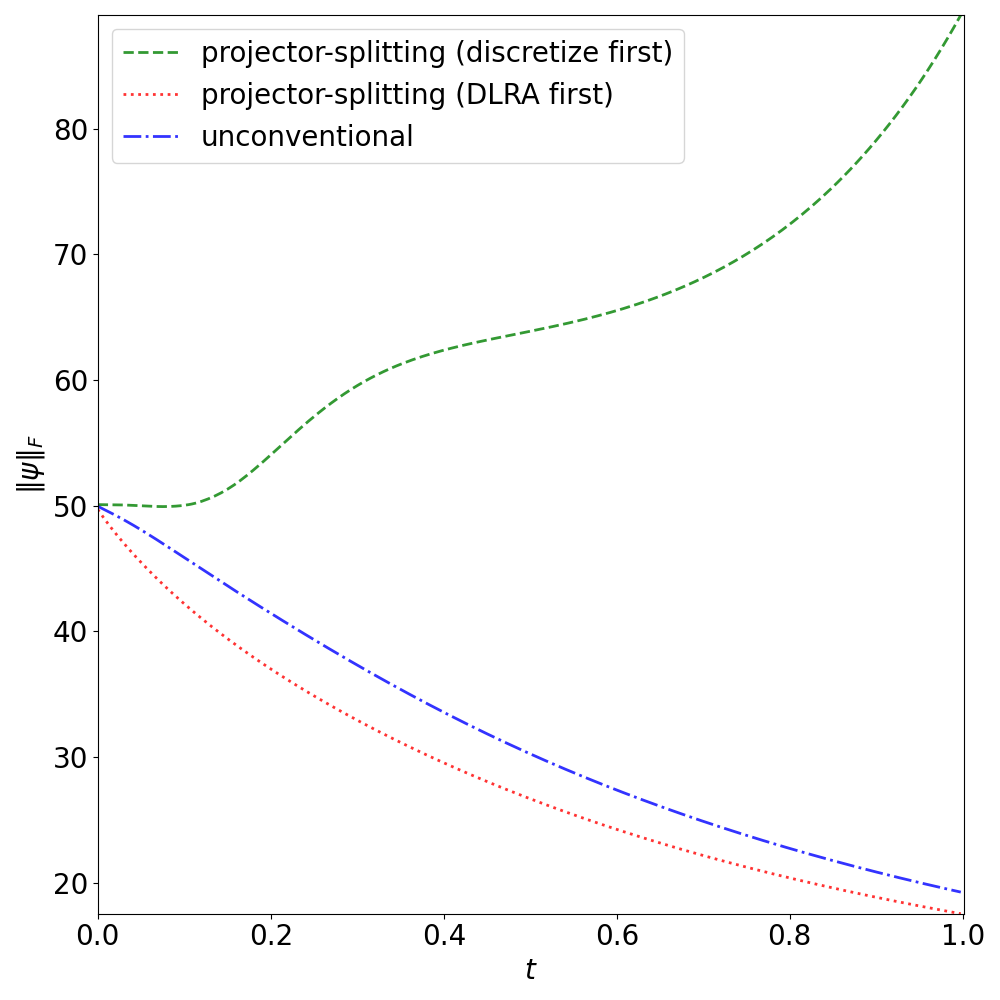}
		\caption{$r=10$}
		\label{fig:studyCFLPNRank10}
	\end{subfigure}
	\begin{subfigure}{0.49\linewidth}
		\centering
		\includegraphics[scale=0.3]{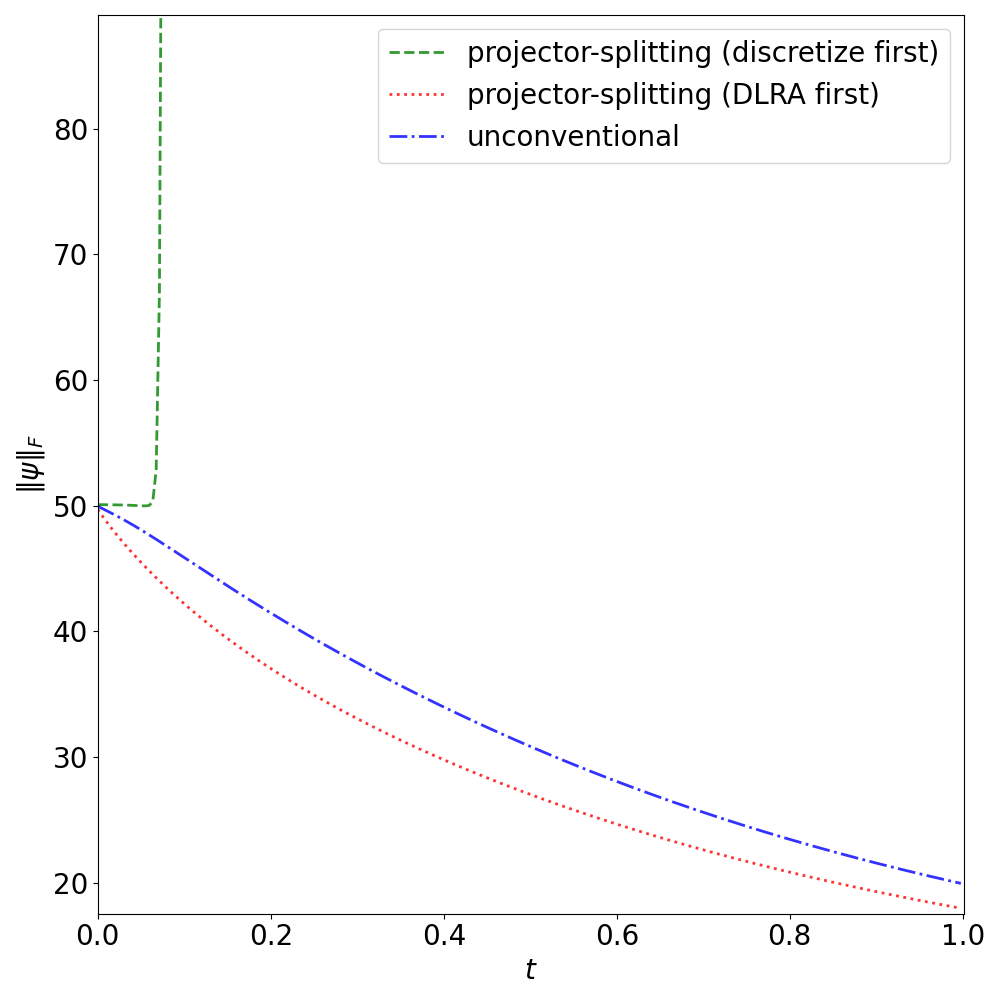}
		\caption{$r=15$}
		\label{fig:studyCFLPNRank15}
	\end{subfigure}
	\caption{Time evolution of the Frobenius norm for ranks $r=10$ and $r=15$ during the computation of the plane source problem.}
	\label{fig:Figure2}
\end{figure}

Lastly, we demonstrate the behavior of the three strategies for varying CFL numbers. For this, we plot the distance of the numerical solution 
\begin{align*}
\phi_{\Delta}(T,x_j) = \sum_{i,m = 1}^r X_i(T,x_j)S_{im}(T) W_{0 m}(T),
\end{align*}
collected in $\mathbf{\phi}_{\Delta}\in\mathbb{R}^{N_x}$ to the analytic reference solution $\mathbf\phi_{\text{ref}} = (\phi_{\text{ref}}(T,x_1),\cdots,\phi_{\text{ref}}(T,x_{N_x}))^T$. The $L^2$-distance is then given by $\Vert \mathbf{\phi}_{\Delta}-\mathbf{\phi}_{\text{ref}} \Vert_{F}$. The behaviour for different CFL numbers when using $r=10$ and $r=15$ is shown in Figure~\ref{fig:Figure3}.
\begin{figure}[h!]
\centering
	\begin{subfigure}{0.49\linewidth}
		\centering
		\includegraphics[scale=0.3]{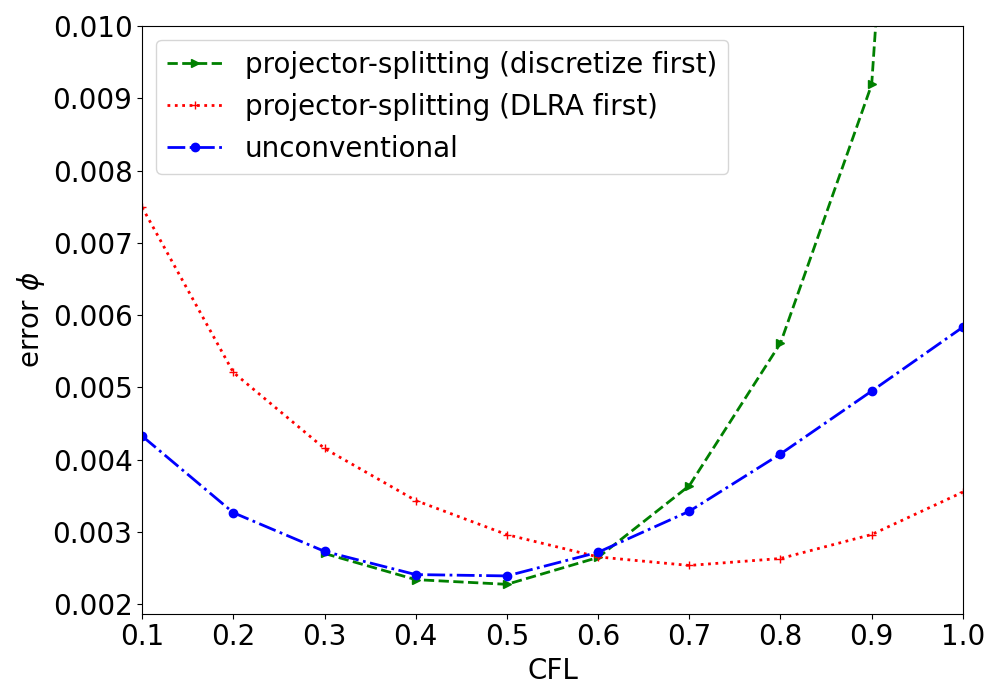}
		\caption{$r=10$}
		\label{fig:Figure3a}
	\end{subfigure}
	\begin{subfigure}{0.49\linewidth}
		\centering
		\includegraphics[scale=0.3]{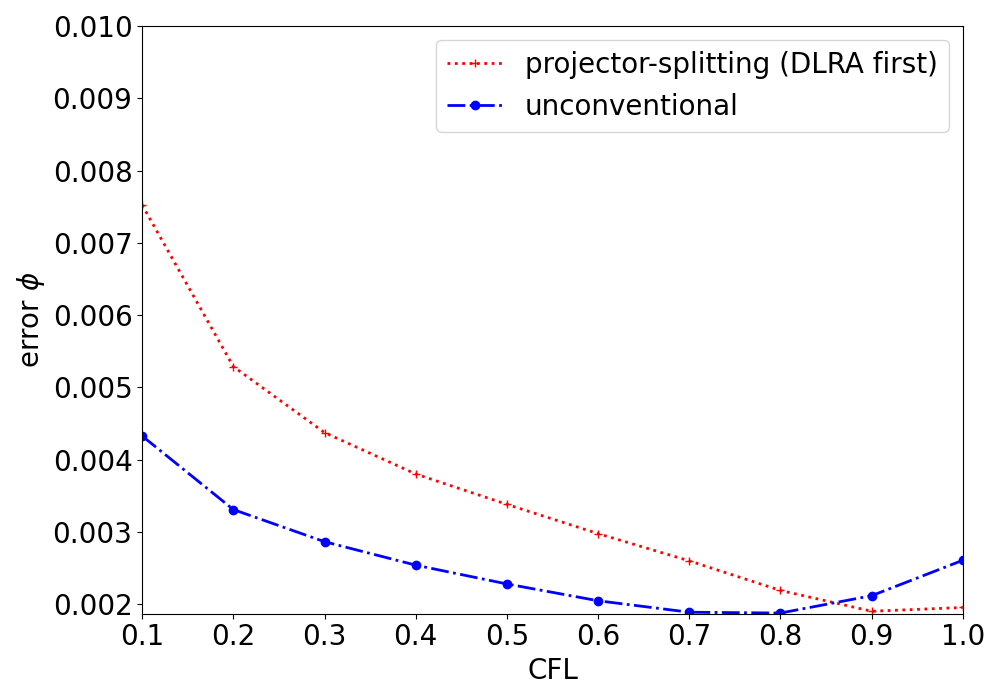}
		\caption{$r=15$}
		\label{fig:Figure3b}
	\end{subfigure}
    \caption{CFL study for the plane source problem.}
	\label{fig:Figure3}
\end{figure}

\subsection{Uncertainty Quantification}
In this section, we investigate the hyperbolic advection equation with uncertain speed (which is also called the random wave equation \cite{gottlieb2008galerkin})
\begin{subequations}\label{eq:advectionEquation}
\begin{align}
    &\partial_t u(t,x,\xi) + a(\xi)\partial_x u(t,x,\xi) = 0,\\
    &u(t=0,x,\xi) = u_{\text{IC}}(x,\xi),\\ 
    &u(t,x_L,\xi) = u_L(t,\xi) \enskip \text{ and } \enskip u(t,x_R,\xi) = u_R(t,\xi).
\end{align}
\end{subequations}
The random variable $\xi$ is uniformly distributed in the interval $[0.2, 1]$ and we choose an uncertain advection speed $a(\xi) = \xi^3$. We pick a deterministic initial condition $u_{\text{IC}}(x,\xi) = \chi_{[-1,0]}$ as well as Dirichlet boundary conditions $u_L = u_R = 0$. A common choice to discretize this system are general polynomial chaos (gPC) basis functions \cite{wiener1938homogeneous,xiu2002wiener}, which in our case are the Legendre polynomials $P_{\ell}$. Then, the solution ansatz takes the form
\begin{align*}
 u(t,x,\xi)\approx \sum_{\ell=0}^N u_{\ell}(t,x)P_{\ell}(\xi).   
\end{align*}
A system of equations describing the time evolution of the gPC expansion coefficients $\mathbf{u} = (u_0,\cdots,u_N)^T$ can be derived with the help of the stochastic-Galerkin (SG) method. Similar to the P$_N$ system, the SG moment system is derived by testing the original problem \eqref{eq:advectionEquation} against the gPC basis functions. The resulting SG system reads
\begin{align}\label{eq:SG}
\partial_t \mathbf u(t,x) = -\mathbf A\partial_x \mathbf u(t,x),
\end{align}
where $\mathbf{A} = (a_{\ell m})_{\ell,m=0}^N$ and $a_{\ell m}=\mathbb{E}[a P_{\ell}P_m]$. Again, a low-rank solution ansatz is chosen and the solution approximation is evolved in time using a dynamical low-rank approximation. In uncertainty quantification one is commonly interested in the standard deviation of the solution, which heavily depends on a finely resolved spatial domain. Therefore, the number of spatial cells is chosen to be $N_x = 2000$. The random domain is discretized with $N+1 = 100$ modal expansion coefficients. All remaining parameter values are chosen as for the radiation transport problem. We start with investigating the solution approximation for a CFL number of one and ranks $5$ and $10$. When deriving the evolution equations of the matrix projector-splitting integrator for the spatially discretized problem at rank $r=5$, we observe an oscillatory approximation for the expectation in Figure~\ref{fig:Figure4a} as well as for the standard deviation in Figure~\ref{fig:Figure5a}. For rank $10$, the projector-splitting integrator for the matrix ODE diverges. The stable discretization of the integrator when deriving the DLRA evolution equations on a continuous level yields finite results for all ranks. Due to its increased dampening compared to the unconventional integrator (cf. Remark~\ref{rem:dampeningPSvsUnconv}), the numerical solution smears out. Improved solution approximations are obtained with the unconventional integrator.
\begin{figure}[h!]
\centering
	\begin{subfigure}{0.49\linewidth}
		\centering
		\includegraphics[scale=0.21]{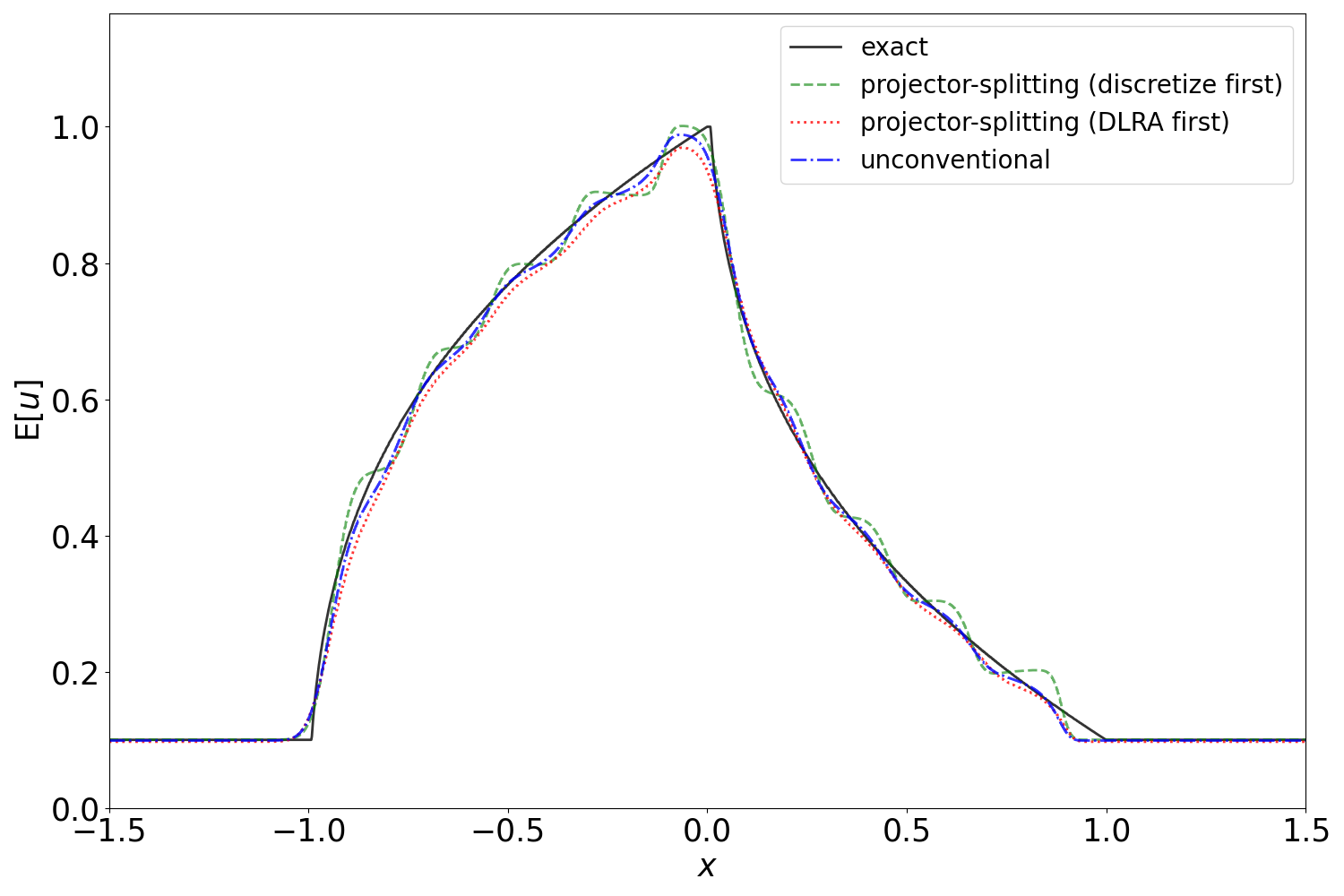}
		\caption{$r=5$}
		\label{fig:Figure4a}
	\end{subfigure}
	\begin{subfigure}{0.49\linewidth}
		\centering
		\includegraphics[scale=0.21]{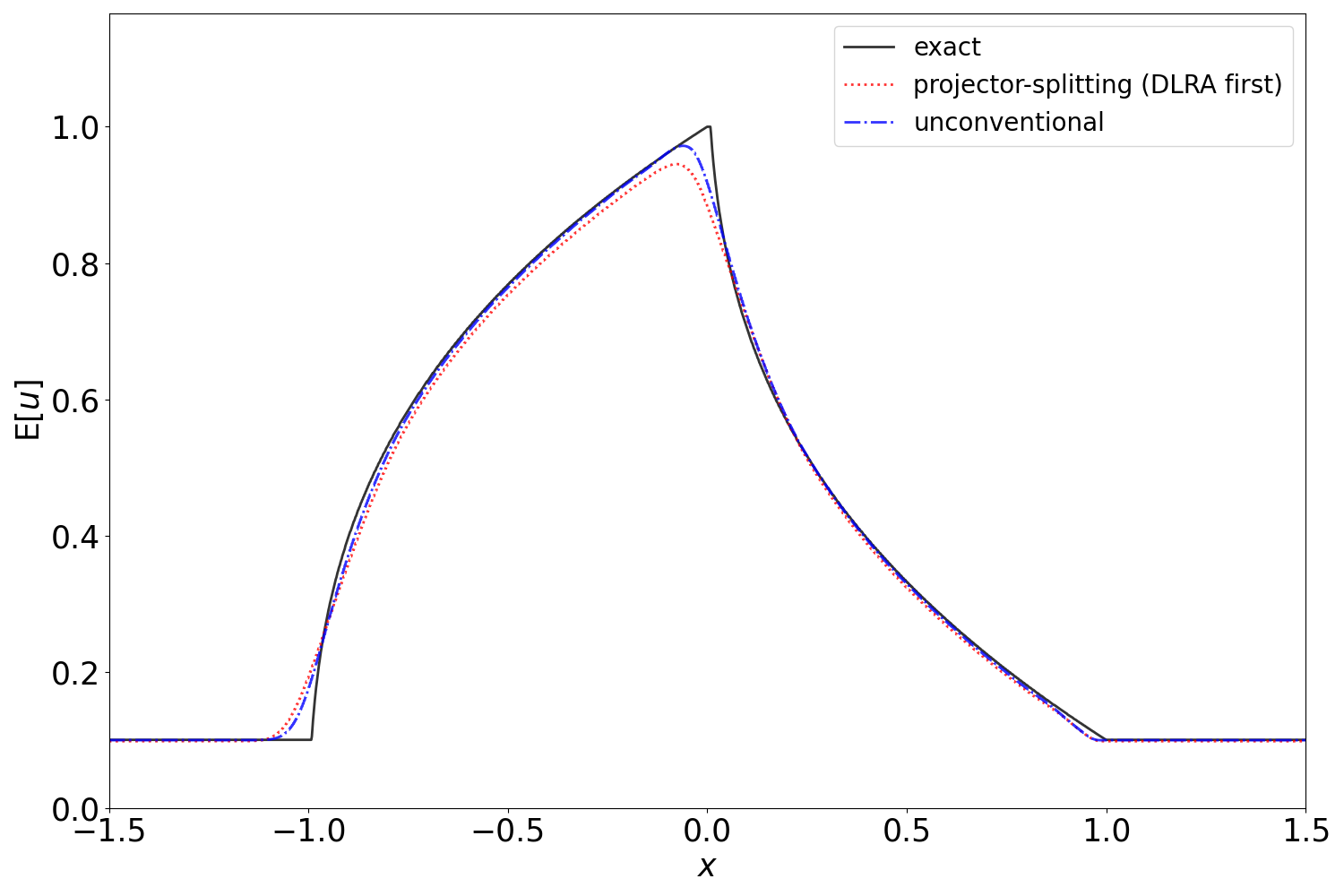}
		\caption{$r=10$}
		\label{fig:Figure4b}
	\end{subfigure}
	\caption{Expected value computed with the unconventional integrator, the projector-splitting integrator and its stabilized discretization for ranks $r=5$ and $r=10$. The term \textit{stable} means that the evolution equations of the integrator are discretized with a stable scheme. The projector-splitting integrator for the fully discretized problem yields infinite values for rank $r=10$ and is therefore not shown.}
	\label{fig:Figure4}
\end{figure}
\begin{figure}[h!]
\centering
	\begin{subfigure}{0.49\linewidth}
		\centering
		\includegraphics[scale=0.21]{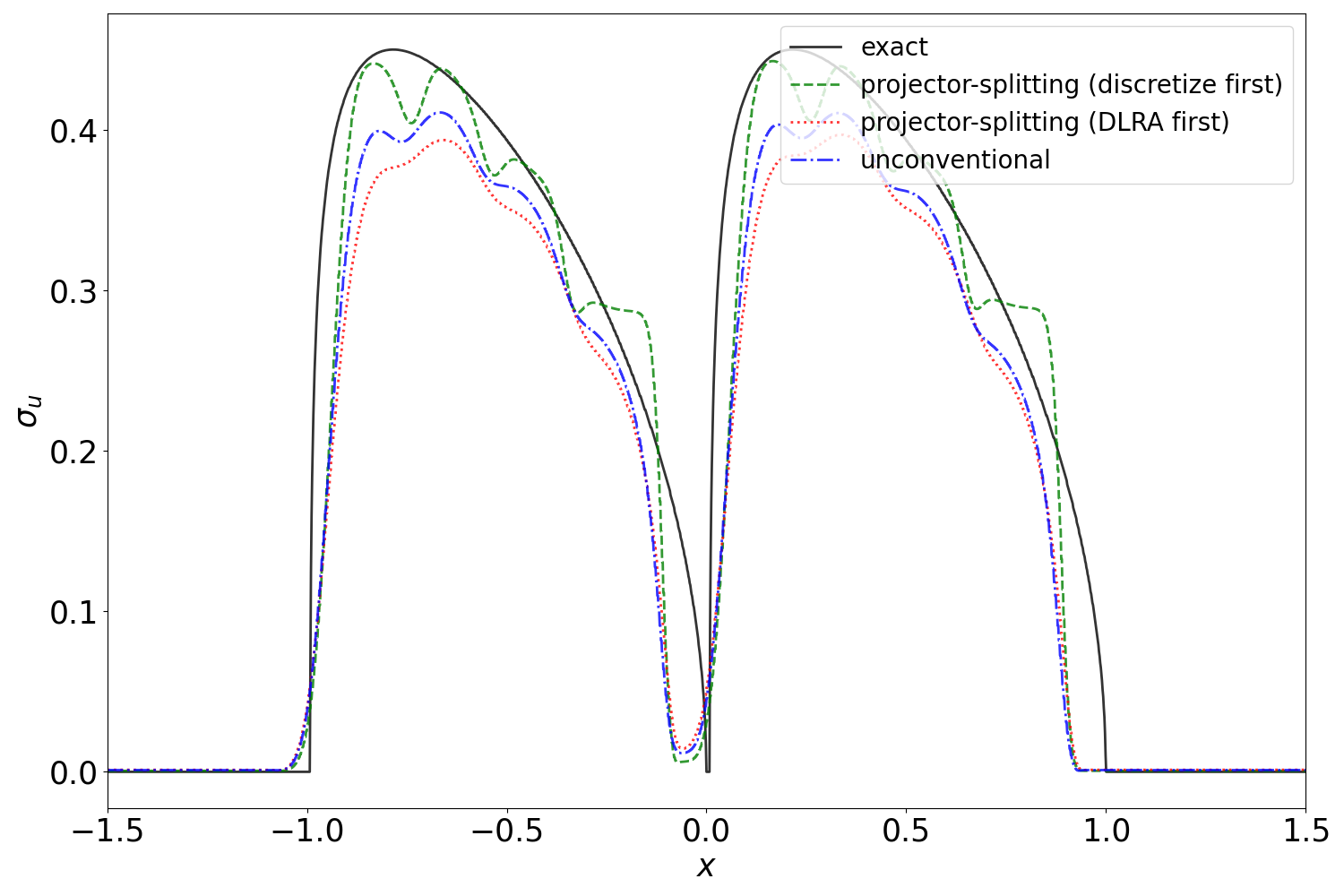}
		\caption{$r=5$}
		\label{fig:Figure5a}
	\end{subfigure}
	\begin{subfigure}{0.49\linewidth}
		\centering
		\includegraphics[scale=0.21]{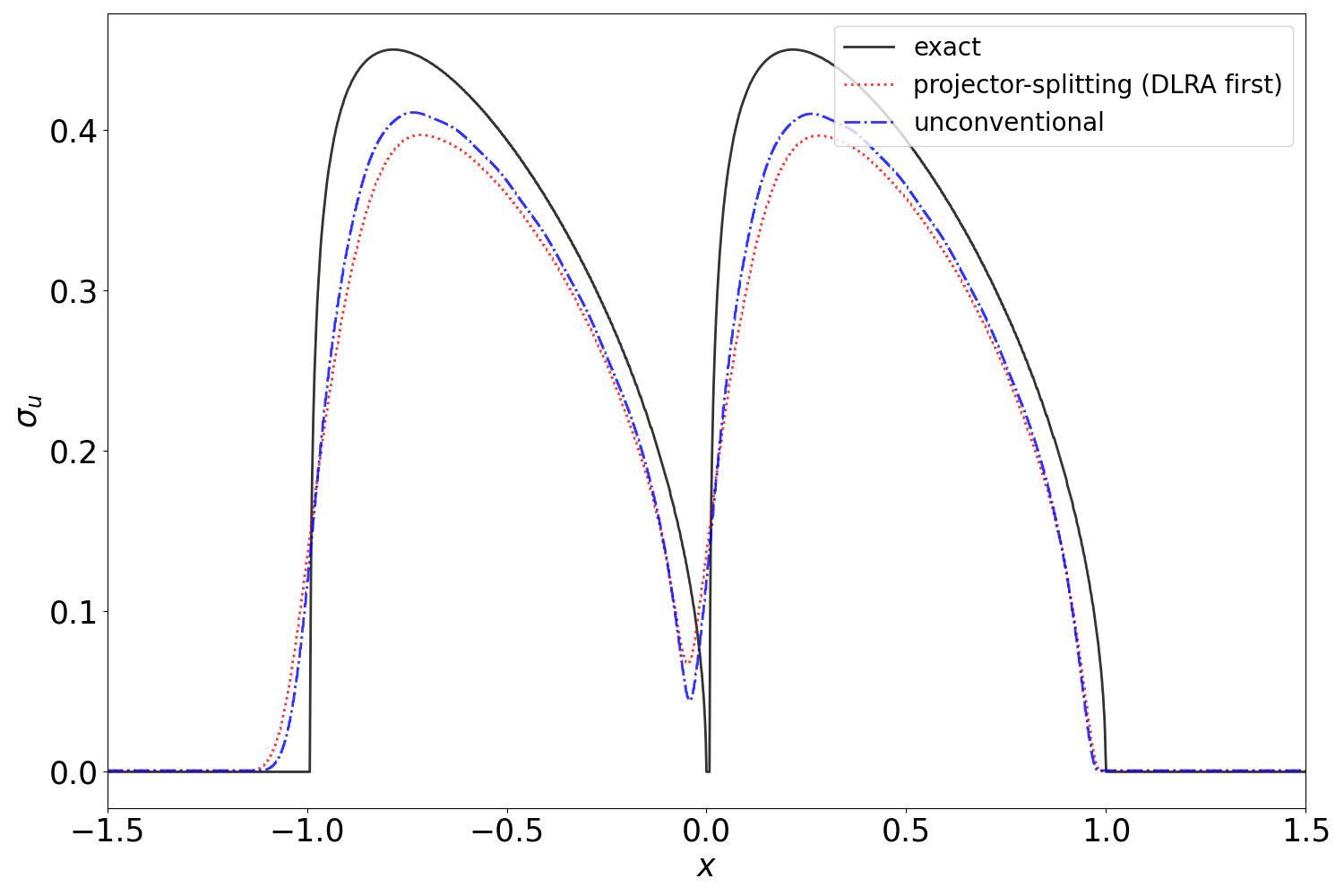}
		\caption{$r=10$}
		\label{fig:Figure5b}
	\end{subfigure}
	\caption{Standard deviation computed with the unconventional integrator, the projector-splitting integrator and its stabilized discretization for ranks $r=5$ and $r=10$. The projector-splitting integrator for the fully discretized problem yields infinite values for rank $r=10$ and is therefore not shown.}
	\label{fig:Figure5}
\end{figure}

Again, the $L^2$-norms of solutions computed with different integrators and discretizations are investigated in Figure~\ref{fig:Figure6}. It is observed that the projector-splitting integrator at rank five leads to a dissipation of the Frobenius norm, whereas the norm is amplified and leads to infinite values at rank $10$. The unconventional integrator and the stable discretization of the projector-splitting integrator both dissipate the norm. As expected, a weaker dissipation is observed for the unconventional integrator.
\begin{figure}[h!]
\centering
	\begin{subfigure}{0.49\linewidth}
		\centering
		\includegraphics[scale=0.3]{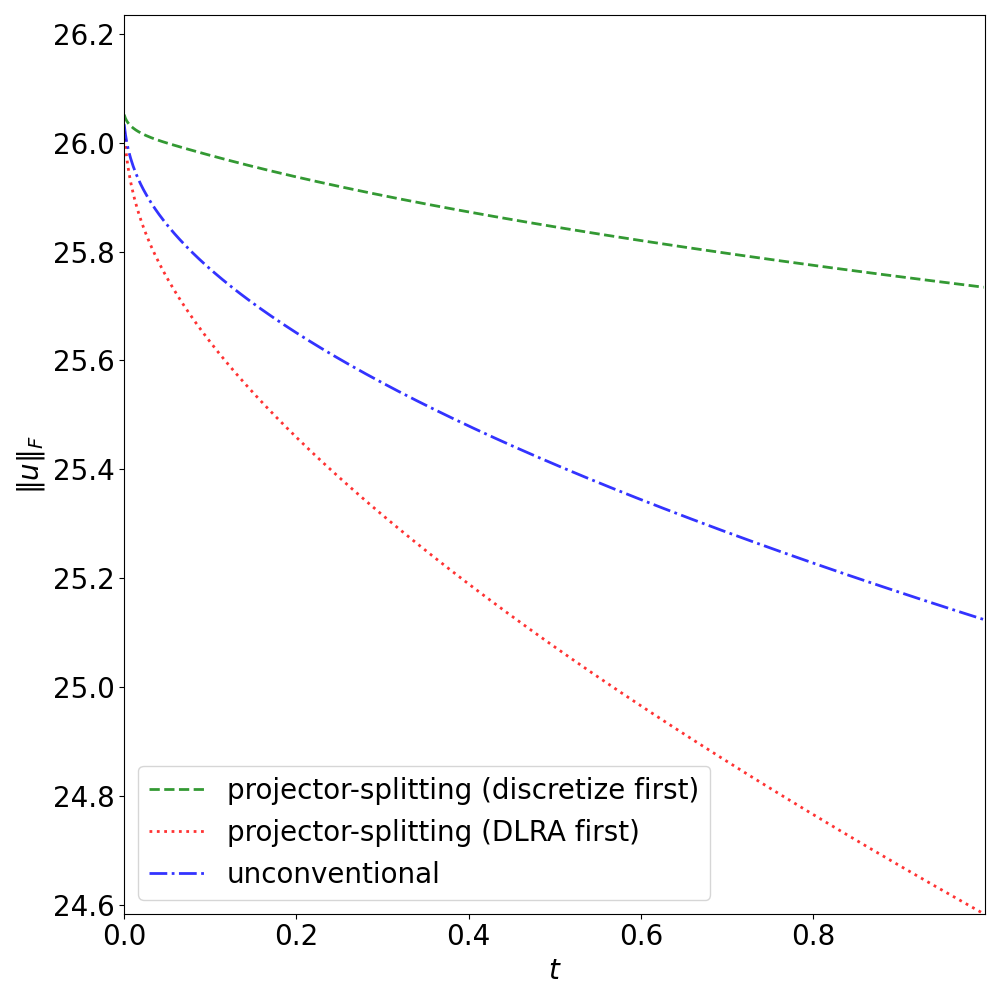}
		\caption{$r=5$}
		\label{fig:Figure6a}
	\end{subfigure}
	\begin{subfigure}{0.49\linewidth}
		\centering
		\includegraphics[scale=0.3]{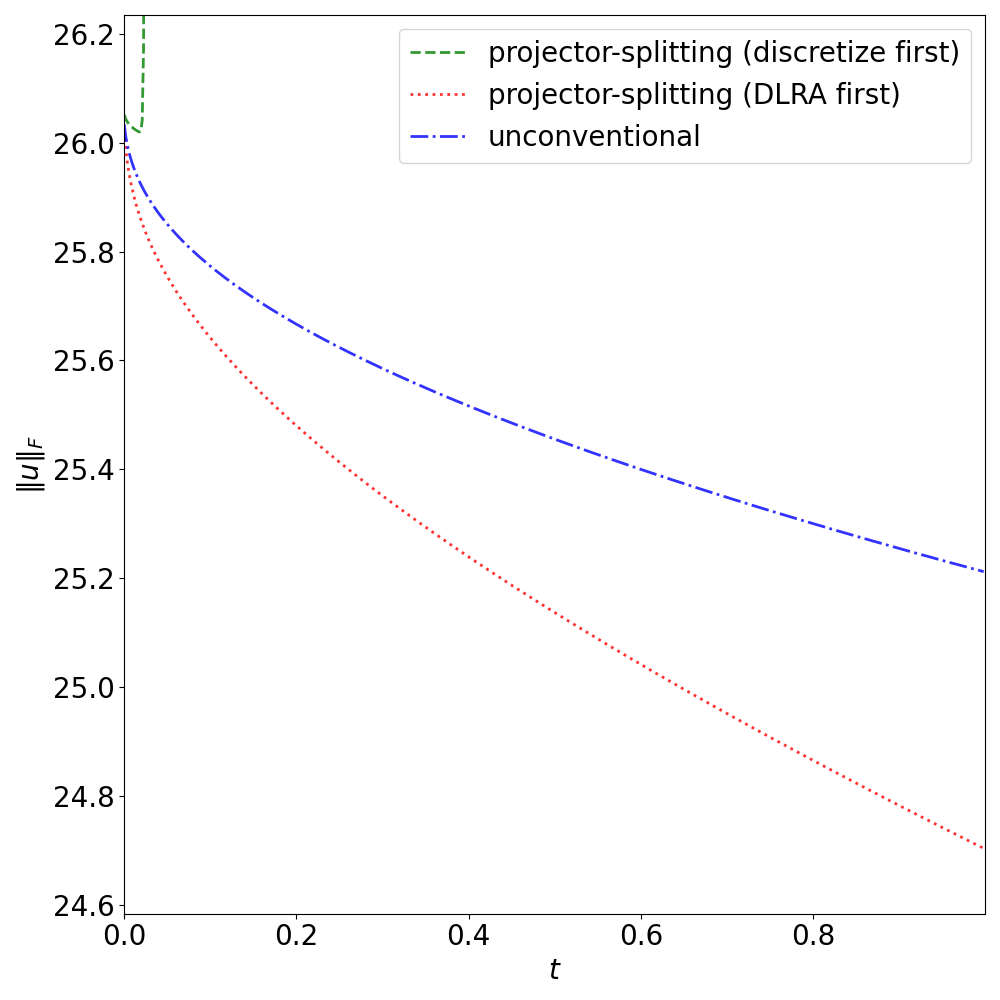}
		\caption{$r=10$}
		\label{fig:Figure6b}
	\end{subfigure}
	\caption{Time evolution of the Frobenius norm for ranks $r=5$ and $r=10$ during the computation of the uncertain advection problem.}
	\label{fig:Figure6}
\end{figure}
\begin{figure}[h!]
\centering
	\begin{subfigure}{0.49\linewidth}
		\centering
		\includegraphics[scale=0.3]{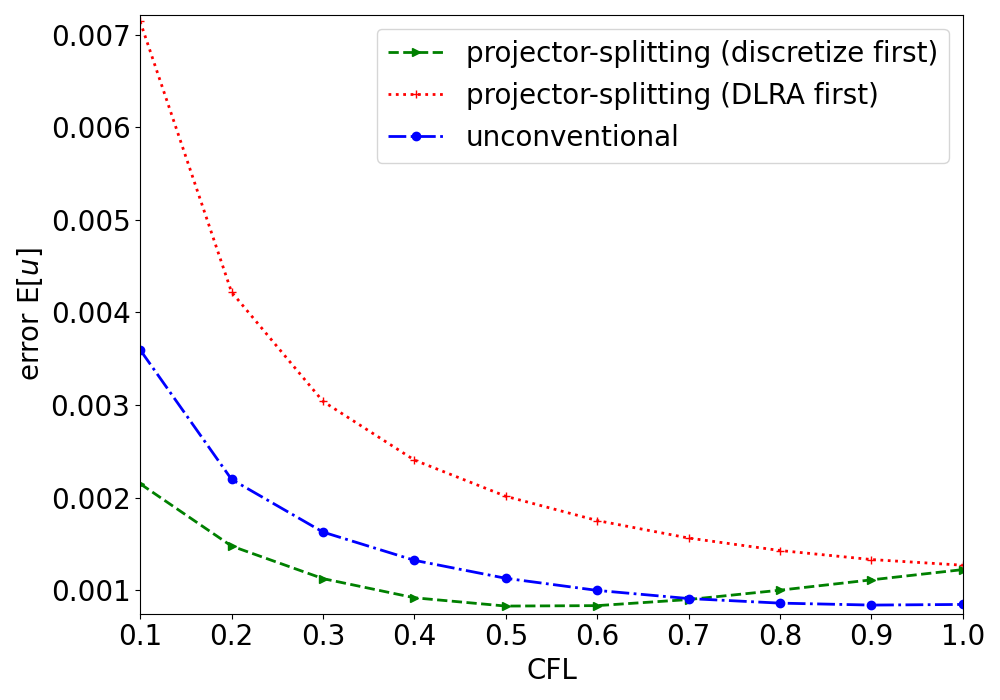}
		\caption{$r=5$, error expectation}
		\label{fig:Figure7a}
	\end{subfigure}
	\begin{subfigure}{0.49\linewidth}
		\centering
		\includegraphics[scale=0.3]{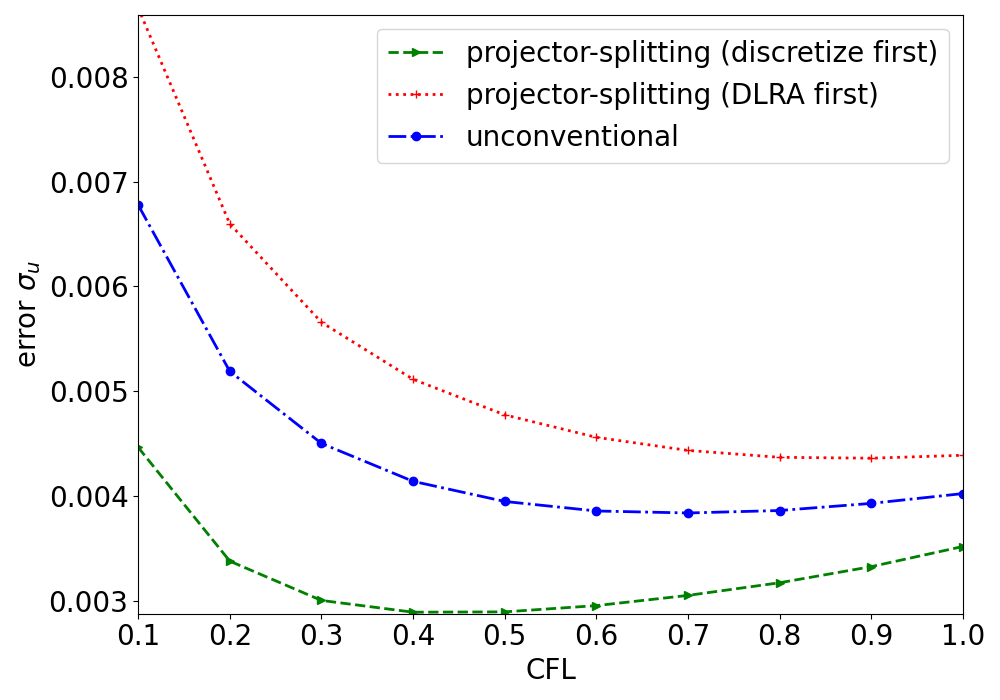}
		\caption{$r=5$, error standard deviation}
		\label{fig:Figure7b}
	\end{subfigure}
		\begin{subfigure}{0.49\linewidth}
		\centering
		\includegraphics[scale=0.3]{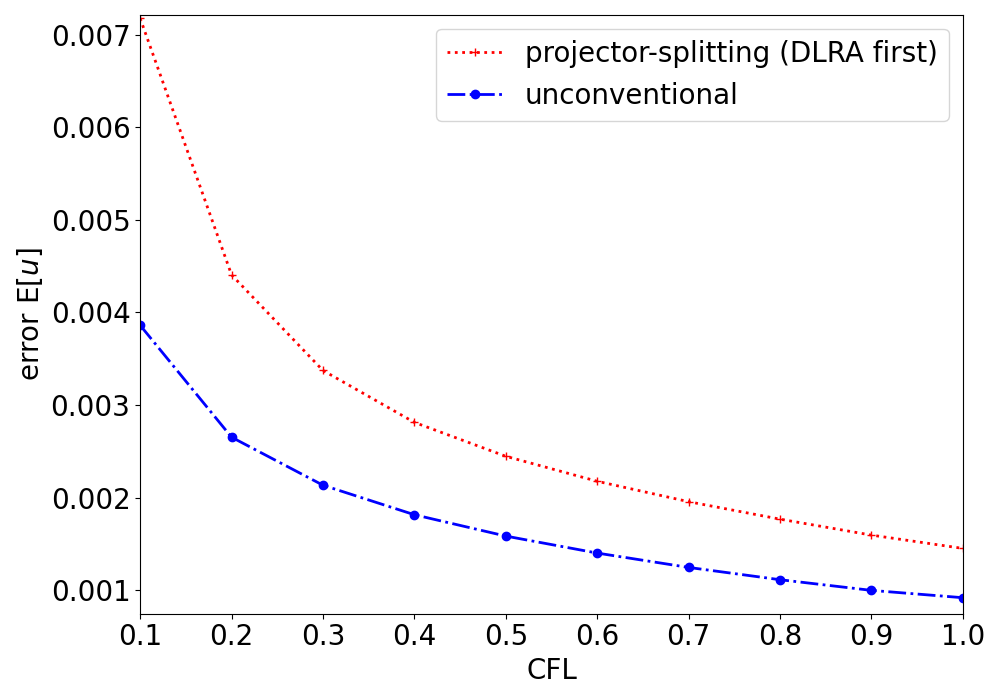}
		\caption{$r=10$, error expectation}
		\label{fig:Figure7a}
	\end{subfigure}
	\begin{subfigure}{0.49\linewidth}
		\centering
		\includegraphics[scale=0.3]{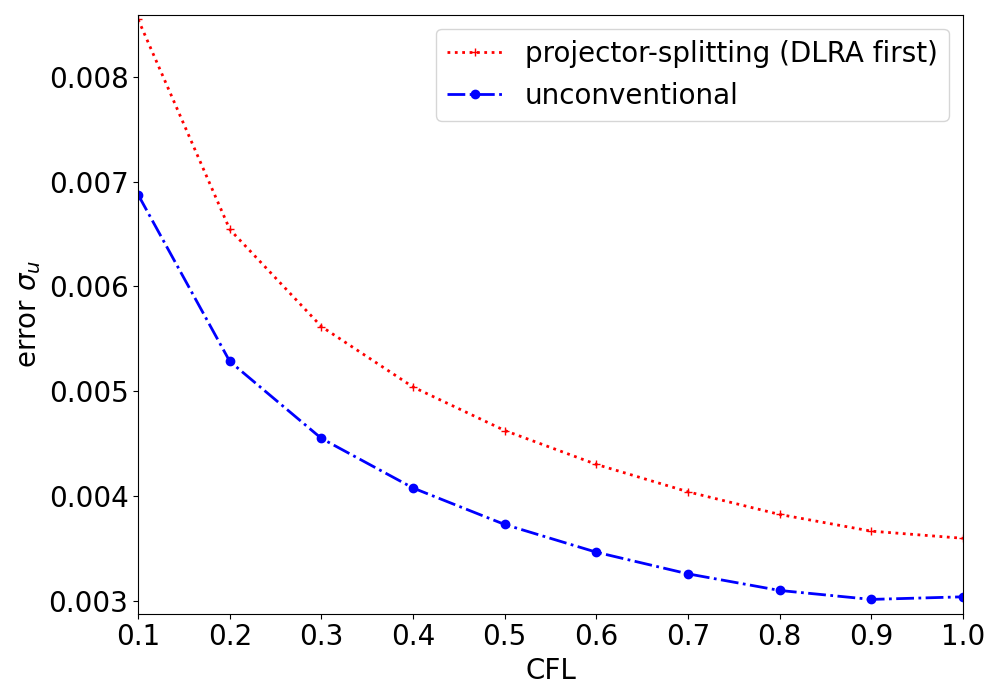}
		\caption{$r=10$, error standard deviation}
		\label{fig:Figure7b}
	\end{subfigure}
	\caption{CFL study for the uncertain advection problem.}
	\label{fig:Figure7}
\end{figure}
Lastly, we perform a CFL study, which we depict in Figure~\ref{fig:Figure7}. Here, we observe that the matrix projector-splitting integrator applied to the discrete system only remains stable for rank $r=5$. For a CFL number of one, the rank five approximation of the expected value shows an increased error compared to the unconventional integrator as well as the stabilized matrix projector-splitting integrator. However, the error of the standard deviation is improved for the projector-splitting integrator, when performing the discretization first. Our analysis provides an idea why this behaviour can be observed. Commonly, higher order moments are strongly affected by artificial diffusion, see e.g. \cite{kusch2020filtered}. Since the $S$-step of the projector-splitting integrator reverts the diffusion which results from the numerical viscosity of the chosen finite volume method, higher order moments are not dampened too heavily. However, the reduced diffusion yields oscillatory approximations of zero order moments, which are commonly improved by artificial diffusion \cite{kusch2020filtered}. Furthermore, as shown in Theorem~\ref{th:L2instability}, reverting diffusion will not guarantee stability, which can be seen for the rank ten results. Here, the projector-splitting integrator applied to the discrete problem becomes unstable. As shown in Theorems~\ref{th:firstDRL} and \ref{th:unconventional}, the unconventional and stabilized projector-splitting integrator guarantee stability. 
%\section{Ideas and Questions}
%\begin{itemize}
%\item Lukas: What happens if we first do the L-step and then the K-step?
%\item perform analysis for Lukas' conservative DLR
%\item derive straight forward expression for cfl condition
%\item Use unstable second order stencil in $K$ and $L$ step for unconventional integrator and in $S$ step for projector-splitting integrator.
%\item TODO: Show effects of high-order RK scheme.
%\item TODO: Study Boltzmann equation, maybe starting from P$_N$ system.
%\item TODO: modified equation analysis
%\end{itemize}

\section*{Acknowledgments}
	The authors would like to thank Christian Lubich for his helpful comments and suggestions, which have been important for the presentation as well as deeper understanding of our stability analysis. This work was funded by the Deutsche Forschungsgemeinschaft (DFG, German Research Foundation) --- Project-ID 258734477 --- SFB 1173.
	
	\bibliographystyle{abbrv}
	\bibliography{stability} 
	
\end{document}